\documentclass{amsart}
\usepackage{amssymb, amsmath, amsthm}
\usepackage[dvipsnames]{xcolor}
\usepackage{hyperref}
\usepackage{stmaryrd}
\usepackage[all]{xy}
\usepackage{tikz-cd}
\usepackage{cleveref}
\usepackage{comment}

\newtheorem{prop}{Proposition}[section]

\newtheorem{lem}[prop]{Lemma}
\newtheorem{cor}[prop]{Corollary}
\newtheorem{thm}[prop]{Theorem}
\theoremstyle{definition}

\theoremstyle{remark}
\newtheorem{examp}[prop]{Example}

\newtheorem{remark}[prop]{Remark}

\DeclareMathAlphabet{\mathpzc}{OT1}{pzc}{m}{it}

\DeclareMathOperator{\Hom}{Hom}

\DeclareMathOperator{\SL}{SL}

\DeclareMathOperator{\Ker}{Ker}

\DeclareMathOperator{\Gal}{Gal}
\DeclareMathOperator{\Image}{Im}

\DeclareMathOperator{\Spec}{Spec}
\DeclareMathOperator{\Proj}{Proj}

\DeclareMathOperator{\codim}{codim}

\DeclareMathOperator{\Br}{Br}
\DeclareMathOperator{\tors}{tors}
\DeclareMathOperator{\Pic}{Pic}
\DeclareMathOperator{\NS}{NS}
\DeclareMathOperator{\Num}{Num}
\DeclareMathOperator{\colim}{colim}

\DeclareMathOperator{\an}{an}
\DeclareMathOperator{\Bl}{Bl}

\newcommand{\Q}{\mathbb{Q}}
\newcommand{\Qp}{\mathbb {Q}_p}

\newcommand{\Qbar}{\overline{\mathbb{Q}}}

\newcommand{\PP}{\mathbb P}

\newcommand{\ZZ}{\mathbb Z}

\newcommand{\QQ}{\mathbb Q}

\newcommand{\Fp}{\mathbb F_p}
\newcommand{\Fq}{\mathbb F_q}
\newcommand{\Fbar}{\overline{\mathbb F}_p}

\newcommand{\red}{\mathrm{red}}

\newcommand{\et}{\operatorname{\acute{e}t}}% etale topology

\usepackage{relsize} 
\usepackage[bbgreekl]{mathbbol} 
\usepackage{amsfonts} 
\usepackage{color}
\DeclareSymbolFontAlphabet{\mathbb}{AMSb} %to ensure that the meaning of \mathbb does not change
\DeclareSymbolFontAlphabet{\mathbbl}{bbold}

\usepackage[T1]{fontenc}  
\usepackage{mathrsfs}

\title{Vanishing of Brauer groups of moduli stacks of stable curves}
\author{Sebastian Bartling}
\address{}
\email{}

\author{Kazuhiro Ito}
\address{Mathematical Institute, Tohoku University, 6-3, Aoba, Aramaki, Aoba-Ku, Sendai 980-8578, Japan}
\email{kazuhiro.ito.c3@tohoku.ac.jp}
\address{Universität Duisburg-Essen, Fakultät für Mathematik, Thea-Leymann Straße, 45127 Essen, Germany}
\email{sebastian-bartling@hotmail.de}

\date{December 2024}

\begin{document}
\begin{abstract}
    We show that the cohomological Brauer groups of
    the moduli stacks of stable genus $g$ curves over the integers and an algebraic closure of the rational numbers vanish for any $g\geq 2$.
    For the $n$ marked version, we show the same vanishing result in the range $(g,n)=(1,n)$ with $1\leq n \leq 6$ and all $(g,n)$ with $g\geq 4.$
    We also discuss several finiteness results on cohomological Brauer groups of proper and smooth Deligne-Mumford stacks over the integers.
    \end{abstract}
\maketitle

\section{Introduction}
The moduli stack $\overline{\mathcal{M}}_{g}$ of stable genus $g \geq 2$ curves as introduced by Deligne-Mumford \cite{DeligneMumfordModuliCurves} has the beautiful property that it is a proper and smooth Deligne-Mumford stack over $\Spec(\ZZ)$.
When it comes to schemes, the few known examples that share this property are typically rational and therefore do not carry non-trivial Brauer classes.
It is then natural to ask whether $\overline{\mathcal{M}}_{g}$ carries non-trivial Brauer classes defined over
$\ZZ$
or
an algebraic closure $\overline{\QQ}$ of $\QQ$.
The main result of this note claims that this is not the case; more precisely, we show the following:

\begin{thm}[see Corollary \ref{Corollary: Brauer group of moduli spaces of curves over Z} and Section \ref{Section: Brauer groups of moduli spaces of stable curves}]\label{thm: main result}
The equalities
$$
\Br(\overline{\mathcal{M}}_{g,\ZZ})=\Br(\overline{\mathcal{M}}_{g,\overline{\Q}})=0
$$
hold for all $g\geq 2$.
    For integers $g, n \geq 0$ such that $2g-2+n>0$, the equalities
$$
\Br(\overline{\mathcal{M}}_{g,n,\ZZ})=\Br(\overline{\mathcal{M}}_{g,n,\overline{\Q}})=0
$$
hold for $(g, n)=(1,n)$ with $1\leq n \leq 6$ and for all $(g,n)$ with $g=0$ or $g\geq 4.$
\end{thm}

Here $\overline{\mathcal{M}}_{g, n}$ stands for the moduli stack of stable genus $g$ curves with $n$ marked points, which is also a proper and smooth Deligne-Mumford stack; see Section \ref{Subsection:Notation and recollections}.
In the following discussion, we also need
the moduli stack
$\mathcal{M}_{g, n} \subset \overline{\mathcal{M}}_{g, n}$
of smooth genus $g$ curves with $n$ marked points.

\begin{remark}
    By our convention to denote by $\Br(-)$ the \textit{cohomological} Brauer group and since the Deligne-Mumford stacks $\overline{\mathcal{M}}_{g,n,\ZZ}$ and $\overline{\mathcal{M}}_{g,n,\overline{\Q}}$ are regular, our result means that
    $$
    H^{2}_{\et}(\overline{\mathcal{M}}_{g,n,\ZZ},\mathbb{G}_{m})=H^{2}_{\et}(\overline{\mathcal{M}}_{g,n,\overline{\Q}},\mathbb{G}_{m})=0
    $$
    in the above range of $(g,n).$
\end{remark}

\begin{remark}\label{Remark:genus zero case}
    The equalities
    $
    \Br(\overline{\mathcal{M}}_{0,n,\ZZ})=\Br(\overline{\mathcal{M}}_{0,n,\overline{\Q}})=0
    $
    for all $n \geq 3$ are well-known.
    For this, we first note that $\overline{\mathcal{M}}_{0,n,\ZZ}$ is a scheme.
    The desired equalities follow (c.f.\ the proof of \cite[Lemma 6.9.8]{Poonen}) from the fact that $\overline{\mathcal{M}}_{0,n,\Q}$
    is a rational proper smooth scheme over $\Q$ and $\overline{\mathcal{M}}_{0,n, \ZZ}$ has a $\ZZ$-section.
    \end{remark}

\begin{remark}\label{Remark: previous studies}
Assume that $g \geq 1$.
To the best of our knowledge, vanishing of $\Br(\overline{\mathcal{M}}_{g,n,\overline{\Q}})$ was known in the following cases of $(g,n)$:
\begin{enumerate}
    \item [(1,1):] Antieau-Meier \cite{AntieauMeierBrauerElliptic} show the stronger statement $\Br(\mathcal{M}_{1,1,\overline{\Q}})=0$.
    (In Proposition \ref{Prop: Vanishing (1,1) bar Q} we give an independent argument.
    We stress that this vanishing is an easy part of \cite{AntieauMeierBrauerElliptic}.)
    \item [(1,2):] It follows from a result of Inchiostro \cite[Theorem 1.2]{InchiostrobarM1_2} that $\Br(\overline{\mathcal{M}}_{1,2, \overline{\QQ}})=0$ (see also Proposition \ref{Prop: vanishing g 1 n 2 3 4 5 6}).
    \item [(2,0):] Here the result follows from \cite{DiLorenzo-PirisiHyperEllEven}; see Proposition \ref{Proposition:genus two case}.
    \item [$g\geq 4:$] Here the result follows from \cite{SecondHomologyGenusGEQ4}; see Theorem \ref{thm vanishing over C}. 
\end{enumerate}
Therefore, the new cases established in this paper are the case $(g, n)=(3, 0)$
and the case
$(g, n)=(1,n)$ with  $3\leq n \leq 6$.

When it comes to vanishing of $\Br(\overline{\mathcal{M}}_{g,n,\ZZ})$ the only known statement that we are aware of concerns the case $(g, n)=(1,1).$ Namely, in the aforementioned article of Antieau-Meier \cite{AntieauMeierBrauerElliptic} they show as their main result that
$$
\Br(\mathcal{M}_{1,1,\ZZ})=0,
$$
which implies $\Br(\overline{\mathcal{M}}_{1,1,\ZZ})=0$ (yet again, our argument for $\Br(\overline{\mathcal{M}}_{1,1,\ZZ})=0$ is independent of their work).
\end{remark}

\begin{remark}\label{Remark:torsion and torsion coefficients}
Since we have
    \[
    \Br(\overline{\mathcal{M}}_{g,n,\overline{\mathbb{Q}}}) = \bigoplus_{\ell} H^3_{\et}(\overline{\mathcal{M}}_{g,n,\overline{\mathbb{Q}}}, \ZZ_\ell(1))_{\mathrm{tor}}
    \]
    (see Lemma \ref{lem: short exact sequence with H3}),
    our result shows that $H^3_{\et}(\overline{\mathcal{M}}_{g,n,\overline{\mathbb{Q}}}, \ZZ_\ell(1))_{\mathrm{tor}}=0$ for any prime $\ell$ in the range of $(g,n)$ stated in Theorem \ref{thm: main result}.
    In fact, together with an extension of a classical result due to Abrashkin and Fontaine (\cite{Abrashkin90}, \cite{FontainePropreLisseZ}) to Deligne-Mumford stacks that are proper and smooth over $\Spec(\ZZ)$ (see Theorem \ref{thm: FontaineAbrashkinDMStacks}),
    this implies that $H^3_{\et}(\overline{\mathcal{M}}_{g,n,\overline{\mathbb{Q}}}, \ZZ_\ell(1))=0$ for any prime $\ell$.
    It would be interesting to investigate whether our main result also extends to the missing cases
    (namely $(1,n)$ with $n\geq 7$ and $(2,n),(3,n)$ with $n\geq 1$), but we currently do not know how to approach these cases.
    Let us also mention here that it is however not true in general that
    $
        H^{3}(\overline{\mathcal{M}}_{g,n,\overline{\mathbb{Q}}},\mathbb{Z}/2\mathbb{Z})=0
    $
    (even for $g \geq 4$).
    This follows from \cite[Theorem 1.2]{EbertGiansiracusaTorsionClass}, but we omit the details here.
\end{remark}

We are not aware of an example of a proper and smooth \textit{scheme} $X$ over $\Spec (\ZZ)$ with non-trivial $\Br(X)$ or $\Br(X_{\overline{\Q}})$ and, as the result of Abrashkin and Fontaine (\cite{Abrashkin90}, \cite{FontainePropreLisseZ}) indicates,
it seems that there are not ``many'' examples of proper and smooth schemes over $\Spec (\ZZ)$.
However,
we have a lot of interesting examples of proper and smooth Deligne-Mumford stacks over $\Spec (\ZZ)$, such as $\overline{\mathcal{M}}_{g,n,\ZZ}$.
We remark that there are quite easy examples of Deligne-Mumford stacks $X$ that are proper and smooth over $\Spec (\ZZ)$, such that
$\Br(X)$ does not vanish
respectively
such that $\Br(X_{\overline{\Q}})$ does not vanish (see Examples \ref{examp: Brauer over integers not zero} and \ref{examp: Brauer over geometric generic fiber not zero}) - in other words, the vanishing patterns proven in this article are not general features of proper and smooth Deligne-Mumford stacks over $\Spec (\ZZ)$; they might be specific to the stack $\overline{\mathcal{M}}_{g,n}$ (and maybe to schemes).

\begin{remark}\label{Remark: finiteness of Brauer group}
    Along the way, we will prove that for Deligne-Mumford stacks $X$ that are proper and smooth over $\Spec (\ZZ)$, the Brauer group $\Br(X)$ is finite; see Theorem \ref{Theorem:finiteness of Brauer group}.
    This may fit into the following conjecture of Artin (\cite[Remarque 2.5 (c)]{GrothendieckGroupedeBrauerIII}): the Brauer group of every proper scheme over $\Spec (\ZZ)$ is finite.
    However, the conjecture of Artin is widely open and is in general much more difficult than the case we are considering here.
\end{remark}

\subsection{Outline of the proof}
Our argument for proving Theorem \ref{thm: main result} proceeds by establishing vanishing of the Brauer group of $\overline{\mathcal{M}}_{g,n,\overline{\Q}}$ and then deducing vanishing over the integers; see Section \ref{Subsection:Deducing vanishing over the integers} for this reduction step, which may be interesting in its own right.

As explained in Remark \ref{Remark: previous studies} (and Remark \ref{Remark:genus zero case}), our results over $\overline{\mathbb{Q}}$ are new for 
$(g, n)=(3, 0)$
and for
$(g, n)=(1,n)$ with  $3\leq n \leq 6$.
Here we outline the proofs for these new cases (but we also give independent arguments for some previous results in the main body of the paper).

Our argument for the case $(g, n)=(3, 0)$ goes as follows:
By the work of Di Lorenzo-Pirisi \cite{Lorenzo-Pirisi}, one knows that for any field $k$ with $p=\text{char}(k)\neq 2,$
$$
\Br(\mathcal{M}_{3, k})=\ZZ/2\ZZ \oplus \Br(k) \oplus B_{p},
$$
where $B_{p}=0$ if $\text{char}(k)=0$ and a $p$-primary torsion group else.
We first reduce the problem of showing that $\Br(\overline{\mathcal{M}}_{3, \overline{\Q}})=0$ to showing that $\Br(\overline{\mathcal{M}}_{3, \overline{\mathbb{F}}_{p}})[2]=0$ for some prime $p\neq 2.$
Our task is therefore to show that the generator $\gamma$ of the factor $\ZZ/2\ZZ$ in $\Br(\mathcal{M}_{3,\mathbb{F}_{q}})[2]$ does not extend to the compactification for a finite field $\mathbb{F}_{q}$ of characteristic $p$.
Using that $\Br(\mathbb{F}_{q}[[t]])=0,$ it is easy to see that it suffices for this to produce a stable curve of genus three $\mathscr{C}$ over $\mathbb{F}_{q}[[t]],$ such that the generic fiber $C=\mathscr{C}_{\eta}$ is a smooth genus three curve over $F=\mathbb{F}_{q}((t))$, which has the property that $\gamma(C) \in \Br(F)[2]$ is non-trivial.
To achieve this, we use the beautiful relation found by Di Lorenzo-Pirisi between the factor $\gamma$ and the étale algebra of 27 lines on smooth cubic surfaces $X$ obtained by blowing up (general) six points in $\mathbb{P}^{2}$, and the result of Shioda \cite{Shioda} on very explicit defining equations of smooth quartic curves which arise as the ramification locus of some projection $X \dashrightarrow \mathbb{P}^{2}$.
The upshot is that for $p\geq 11$ we are able to find such curves; see Section \ref{Subsection: genus three} for details.

For the cases $(1,n)$ with $3\leq n \leq 6$
(also for $n=1, 2$), we give a uniform argument; see Proposition \ref{Prop: vanishing g 1 n 2 3 4 5 6}.
Namely, we use recent work of Battistella-Di Lorenzo \cite{battistella2024wallcrossingintegralchowrings} (relying on work of Smyth \cite{Smyth19DiagramG1} and Lekili-Polishchuk \cite{LP19ModularCompactification}), who explain that in these cases the stack $\overline{\mathcal{M}}_{1,n}$ is connected to a stacky weighted projective space (resp. Grassmannian) by a diagram of smooth Deligne-Mumford stacks along weighted (and ordinary) blow-ups. 
Vanishing of the Brauer group of a stacky weighted projective space over $\overline{\mathbb{Q}}$ follows from the work of Shin \cite{ShinBrauerWeighted} and we deduce the desired result by using purity of the Brauer group (see Proposition \ref{Prop: purity Brauer}) as established by \v{C}esnavi\v{c}ius \cite{CesnaviciusPurityBrauer}.

\subsection{Structure of the article}
In Section \ref{Section:Preparations} we collect some results which are extensions of well-known results in the context of schemes to the setting of Deligne-Mumford stacks, which will be used later. In Section \ref{section: General results on Brauer groups} we first show finiteness of the Brauer group of a proper and smooth Deligne-Mumford stack $X$ over $\Spec(\ZZ)$, see Theorem \ref{Theorem:finiteness of Brauer group}, then we explain how in this setting it is often possible
(in particular, possible for $X=\overline{\mathcal{M}}_{g,n}$) to reduce vanishing of $\Br(X)$ to vanishing of $\Br(X_{\overline{\mathbb{Q}}})$,
see Proposition \ref{Proposition: deducing vanishing over Z}. We end that section with a discussion of some examples. Finally, in Section \ref{Section: Brauer groups of moduli spaces of stable curves} we prove our vanishing results on $\Br(\overline{\mathcal{M}}_{g,n, \overline{\mathbb{Q}}})$.

\subsection{Conventions, notations and recollections}\label{Subsection:Notation and recollections}
All Deligne-Mumford (DM) stacks considered in this paper will be assumed to be quasi-separated.
If $X$ is a scheme or DM stack, we denote by
\[
\Br(X)=H^{2}_{\et}(X,\mathbb{G}_{m})_{\tors}
\]
the cohomological Brauer group.
Here, for an abelian group $M$, we let $M_{\tors}$ denote its torsion part.
Recall that $\Br(X)=H^{2}_{\et}(X,\mathbb{G}_{m})$ if $X$ is noetherian and regular (see \cite[Proposition 2.5]{AntieauMeierBrauerElliptic}).
Recall here that a DM stack $X$ is called regular if there exists an étale presentation of $X$ by a regular scheme (then all presentations will be by regular schemes). 
We use the notion of the codimension of a closed substack of a noetherian DM stack as in \cite[Definition 6.1]{osserman2013relativedimensionmorphismsdimension}.

For integers $g, n \geq 0$ such that $2g-2+n>0$,
we denote by $\mathcal{M}_{g,n}$ the moduli stack of smooth genus $g$ curves with $n$ marked points, by $\overline{\mathcal{M}}_{g,n}$ the moduli stack of stable genus $g$ curves with $n$ marked points; see \cite[Section 1]{Knudsen} for example for the basic definition.
If $n=0$ (so that $g \geq 2$), then we write
$\mathcal{M}_{g}:=\mathcal{M}_{g,0}$
and
$\overline{\mathcal{M}}_{g}:=\overline{\mathcal{M}}_{g,0}$.
It is well-known that
$\overline{\mathcal{M}}_{g,n}$
is a proper and smooth DM stack over $\Spec (\ZZ)$ with geometric connected fibers 
and $\mathcal{M}_{g,n} \subset \overline{\mathcal{M}}_{g,n}$ is open and dense.
See \cite{DeligneMumfordModuliCurves} and \cite{Knudsen}.
For a ring $R$, we denote by
$\mathcal{M}_{g,n, R}$
and
$\overline{\mathcal{M}}_{g,n, R}$
the base changes to $\Spec (R)$.
We implicitly always assume that $2g-2+n>0$ when considering these moduli stacks.

\section{Preparations}\label{Section:Preparations}
We start by collecting some statements we need for our arguments later.

\begin{lem}\label{lem: blow up regular center smooth DM stack}
    Let $k$ be a field, $X$ a smooth DM stack over $\Spec(k),$ $Z\hookrightarrow X$ a closed DM substack, which is also smooth over $\Spec(k).$
    Consider the blow-up $\Bl_{Z}(X)$ of $X$ along $Z$ (see e.g.\ \cite[Example 3.2.6]{WeightedBlowUps} for a definition).
    Then $\Bl_{Z}(X)$ is smooth over $\Spec(k).$
\end{lem}
\begin{proof}
    Let $U\rightarrow X$ be an étale presentation of $X$ by a scheme $U.$ Since blowing up commutes with flat localization on the base (\cite[Corollary 3.2.14 (iii)]{WeightedBlowUps}), we obtain the étale presentation
    $$
        \Bl_{Z_{U}}(U)\rightarrow \Bl_{Z}(X)
    $$
    and it suffices by the definition of a smooth morphism of algebraic stacks (\cite[Tag 075U]{stacks}) to show that $\Bl_{Z_{U}}(U)$ is a smooth scheme over $\Spec(k).$
    But then $Z_{U}$ is Zariski locally cut out in $U$ by a regular sequence and the result is classical.
\end{proof}

We next deduce the following purity statement for the Brauer group of a noetherian and regular DM stack from results of \v{C}esnavi\v{c}ius \cite{CesnaviciusPurityBrauer}; this is stronger than what we will later need but we wanted to record it here.

\begin{prop}\label{Prop: purity Brauer}
    Let $X$ be a noetherian and regular DM stack and $Z\hookrightarrow X$ a closed substack such that $\codim_{X}(Z)\geq 2,$ then the natural restriction map
    $$
    H^{i}_{\et}(X,\mathbb{G}_{m})\rightarrow H^{i}_{\et}(X-Z,\mathbb{G}_{m})
    $$
    is an isomorphism for $i \leq 2$ and is injective for $i=3$.
\end{prop}

\begin{proof}
As explained in the proof of \cite[Theorem 6.1]{CesnaviciusPurityBrauer}, we can work \'etale locally and reduce the desired statement to the case where $X$ is a scheme (or more precisely, the spectrum of a strictly henselian regular local ring), and thus this proposition follows from the result proved there.
\begin{comment}
We first suppose that $X$ is in fact an algebraic space.
    Let $U\rightarrow X$ be an étale presentation by a scheme $U$ and denote by $Z_{U}\hookrightarrow U$ the pullback; observe that $U$ is a regular scheme and $Z_{U}$ is still of codimension greater than or equal to two in $U$.
    The same applies to $U^{i/_{X}}=U\times_{X}...\times_{X}U$ ($i$-times) and the closed subscheme $Z_{U^{i/X}}\hookrightarrow U^{i/_{X}}.$
    We recall that $\mathbb{G}_{m}(U^{i/_{X}}-Z_{U^{i/X}})=\mathbb{G}_{m}(U^{i/_{X}})$ (\cite[Tag 031T]{stacks}) and $\Pic(U^{i/_{X}}-Z_{U^{i/X}})=\Pic(U^{i/_{X}})$ by regularity.
    By the purity result for the Brauer group established by \v{C}esnavi\v{c}ius \cite{CesnaviciusPurityBrauer}, we likewise obtain
    $$
    H^{2}_{\et}(U^{i/_{X}}-Z_{U^{i/X}},\mathbb{G}_{m})=H^{2}_{\et}(U^{i/_{X}},\mathbb{G}_{m}).
    $$
    We have two convergent cohomological descent spectral sequences starting on the $E_{1}$-page:
    $$
    E_{1}^{i,j}=H^{j}_{\et}(U^{i/_{X}},\mathbb{G}_{m}) \Rightarrow H^{i+j}_{\et}(X,\mathbb{G}_{m})
    $$
    and
    $$
    F_{1}^{i,j}=H^{j}_{\et}(U^{i/_{X}}-Z_{U^{i/X}},\mathbb{G}_{m}) \Rightarrow H^{i+j}_{\et}(X-Z,\mathbb{G}_{m}).
    $$
    For $j\leq 2$ and all $i,$ we explained that
    $
    E_{1}^{i,j}=F_{1}^{i,j},
    $
    which implies that $E_{k}^{i,j}=F_{k}^{i,j}$ $(i+j=2)$ on each page.
    Hence the result follows.
    In general, one may now rerun the same argument to deduce also the case when $X$ is a DM stack.
\end{comment}
\end{proof}

\begin{cor}\label{cor: Brauer groups blow ups}
    Let $\pi\colon \widetilde{X}\rightarrow X$ be a morphism of noetherian and regular DM stacks which is an isomorphism on a dense open substack $U\subset X.$ 
    \begin{enumerate}
        \item The homomorphism $\pi^{*} \colon \Br(X)\rightarrow \Br(\widetilde{X})$ is injective.
        \item If the complement $Z$ of $U$ in $X$ satisfies $\codim_{X}(Z)\geq 2,$ then if $\Br(X)=0,$ also $\Br(\widetilde{X})=0.$ 
    \end{enumerate}
\end{cor}
\begin{proof}
    The first part follows because both $\Br(X)$ and $\Br(\widetilde{X})$ embed into $\Br(U)$ (\cite[Proposition 2.5 (iv)]{AntieauMeierBrauerElliptic}). For the second part, purity of the Brauer group (Proposition \ref{Prop: purity Brauer}) implies that $\Br(U)=0,$ so that by \cite[Proposition 2.5 (iv)]{AntieauMeierBrauerElliptic} again, we see that $\Br(\widetilde{X})=0,$ as desired.
\end{proof}

Next we want to record some extensions to the setting of DM stacks of statements in étale cohomology which are classical for schemes; these results are most probably well-known to the experts.

\begin{comment}
    Next we want to record some extensions to the setting of DM stacks of statements in étale cohomology which are classical for schemes; these results are most probably well-known to the experts but we could not find a reference.
\end{comment}

\begin{lem}\label{lem: Galois rep is everywhere unramified and crystalline at p}
    Let $X\rightarrow \Spec(\ZZ)$ be a proper and smooth DM stack.
    The $\Gal(\overline{\Q}/\Q)$-representation 
    $$
    H^{i}_{\et}(X_{\overline{\Q}},\mathbb{Q}_{p})
    $$
    is unramified at $\ell\neq p$ and crystalline at $p.$
\end{lem}
\begin{proof}
    The claim about unramifiedness at $\ell\neq p$ follows from \cite[Proposition 3.1]{BogaartEdixhoven} and about being crystalline at $p$ from \cite[Theorem 4.3.25, Remark 4.1.15]{KubrakPrikhodkopHodgeStacks}.
\end{proof}

\begin{lem}\label{lem: trace morphism for generically finite etale morphism}
Let $k$ be a finite field or an algebraically closed field.
Let $\ell$ be a prime invertible in $k$.
Let $f \colon Y \to X$ be a morphism of separated DM stacks over $\Spec (k)$.
We assume that $X$ and $Y$ are smooth pure of dimension $d$ over $\Spec (k)$.
We further assume that $f$ is representable by schemes and proper, and that $f$ is finite \'etale of constant degree $m$ over a dense open substack of $X$.
We set $\Lambda :=\ZZ/\ell^n \ZZ$, and let $\Lambda_X$ and $\Lambda_Y$ denote the corresponding constant sheaves on $X$ and $Y$, respectively.
\begin{enumerate}
    \item The dualizing complexes of $X$ and $Y$ (in the sense of \cite{Laszlo-Olsson}) are isomorphic to
    $\Lambda_X(d)[2d]$ and 
    $\Lambda_Y(d)[2d]$, respectively.
    We have a canonical trace morphism
    \[
    \mathrm{tr}_f \colon R f_* \Lambda_Y(d)[2d] \to \Lambda_X(d)[2d].
    \]
    \item The composition
\[
\Phi \colon \Lambda_X(d) \to
R f_*\Lambda_Y(d)
\overset{\mathrm{tr}_f[-2d]}{\longrightarrow} \Lambda_X(d)
\]
is equal to the multiplication by $m$.
\end{enumerate}
\end{lem}

\begin{proof}
    (1) The first assertion follows from \cite[Corollary 4.6.2]{Laszlo-Olsson}.
    We recall that the functors $Rf_!$ and $f^!$ are defined in \cite[Definition 4.4.1]{Laszlo-Olsson}, which satisfy the usual adjunction property; see \cite[Proposition 4.4.2]{Laszlo-Olsson}.
    We have $f^!\Lambda_X(d)[2d] = \Lambda_Y(d)[2d]$ (see \cite[(4.4.i)]{Laszlo-Olsson}).
    By \cite[Proposition 5.2.1]{Laszlo-Olsson}, we have
    $
    Rf_!=Rf_*.
    $
    Then $\mathrm{tr}_f$ is defined as the following composition
    \[
    R f_* \Lambda_Y(d)[2d] = Rf_!\Lambda_Y(d)[2d] = Rf_!f^!\Lambda_X(d)[2d] \to \Lambda_X(d)[2d]
    \]
    where the last morphism is induced by adjunction.

    (2) We may work \'etale locally on $X$.
    We may thus assume that $X$ is a scheme and $k$ is algebraically closed.
    Then $Y$ is also a scheme.
    It suffices to show that
    the induced homomorphism
    \[
    H^0_{\et}(X, \Lambda(d)) \overset{\Phi}{\to} H^0_{\et}(X, \Lambda(d))
    \]
    is equal to the multiplication by $m$.
    Let $U \subset X$ be an open dense subscheme such that $f$ is finite \'etale over $U$.
    Since we have
    $H^0_{\et}(X, \Lambda(d))=H^0_{\et}(U, \Lambda(d))$,
    the assertion now follows from the corresponding statement for the trace morphisms associated with finite \'etale morphisms of schemes.
\end{proof}

Lemma \ref{lem: trace morphism for generically finite etale morphism} enables us to deduce the Riemann Hypothesis part of Weil conjectures for DM stacks.

\begin{prop}\label{prop: Riemann Hypothesis Weil DM stacks}
    Let
    $X$ be a proper and smooth DM stack over $\mathbb{F}_{q}$.
    The eigenvalues of the geometric Frobenius $\varphi_{q}\in \Gal(\overline{\mathbb{F}}_{q}/\mathbb{F}_{q})$ acting on
    $H^{i}_{\et}(X_{\overline{\mathbb{F}}_{q}},\mathbb{Q}_{\ell})$
    are algebraic integers with complex absolute value $q^{i/2}$.
\end{prop}
\begin{proof}
    We may assume that $X$ is connected.
    Since $X$ is smooth, it follows that $X$ is irreducible.
    By \cite[Th\'eor\`eme 16.6]{Laumon-Moret-Bailly}, there exists a finite, surjective, generically \'etale morphism
    $Z \to X$ with $Z$ a scheme.
    After replacing $Z$ by an irreducible component of it (with the reduced induced structure) which dominates $X$, we may assume that $Z$ is integral.
    Since $\mathbb{F}_{q}$ is perfect,
    by \cite[Theorem 4.1]{deJongAlteration}, there exists a proper and generically finite \'etale morphism
    $Y \to Z$
    such that $Y$ is a connected proper smooth scheme over $\mathbb{F}_{q}$.
    The composition 
    $f \colon Y \to X$ satisfies the conditions in Lemma \ref{lem: trace morphism for generically finite etale morphism}.
    Thus, the trace morphism induces a homomorphism
    \[
    H^i_{\et}(Y_{\overline{\mathbb{F}}_{q}}, \QQ_\ell) \to H^i_{\et}(X_{\overline{\mathbb{F}}_{q}}, \QQ_\ell)
    \]
    such that the composition
    \[
    H^i_{\et}(X_{\overline{\mathbb{F}}_{q}}, \QQ_\ell) \overset{f^*}{\to} H^i_{\et}(Y_{\overline{\mathbb{F}}_{q}}, \QQ_\ell) \to H^i_{\et}(X_{\overline{\mathbb{F}}_{q}}, \QQ_\ell)
    \]
    is equal to the multiplication by some positive integer, and hence is an isomorphism.
    It follows that
    $H^i_{\et}(X_{\overline{\mathbb{F}}_{q}}, \QQ_\ell)$
    is isomorphic to a direct summand of
    $H^i_{\et}(Y_{\overline{\mathbb{F}}_{q}}, \QQ_\ell)$
    as a representation of $\Gal(\overline{\mathbb{F}}_{q}/\mathbb{F}_{q})$.
    Therefore the desired assertion follows from the Weil conjecture for schemes proved by Deligne \cite{WeilI, WeilII}.
\end{proof}

We also record the following statement which will facilitate the passage between statements concerning the Brauer group over $\overline{\Q}$ and $\mathbb{C}.$

\begin{lem}\label{lem: short exact sequence with H3}
    Let $K$ be an algebraically closed field of characteristic $0$ such that $K\subset \mathbb{C}.$
    Let $X\rightarrow \Spec(K)$ be a connected proper smooth DM stack. 
    There is an exact sequence
    \[
    0 \to \Br(X)^{0} \to \Br(X) \to H^{3}(X_{\mathbb{C}}^{\an},\ZZ)_{\tors} \to 0,
    \]
    where $\Br(X)^{0}\simeq (\Q/\ZZ)^{b_{2}-\rho}$ with $b_{2}(X)=\dim_{\Q_{\ell}}H^{2}_{\et}(X, \Q_{\ell})$ and $\rho(X)$ is the rank of the N\'eron-Severi group $\NS(X)$.
    Here $X_{\mathbb{C}}^{\an}$ is the complex analytification of $X_{\mathbb{C}}$.
\end{lem}
\begin{proof}
    For schemes this is \cite[Proposition 4.2.6 (ii), (iii)]{BrauerGrothendieckBook}. 
    Let us quickly explain how to adapt this to the setting here.
    
    By \cite[Théorème 4.3.1]{BrochardPicardFunctorStack},
    the identity component
    $\underline{\Pic}^{\circ}_{X/K}$ is a connected proper group scheme over $K$, so that $\underline{\Pic}^{\circ}_{X/K,\red}$ is an abelian variety over $K$.
    (In fact, since $K$ is characteristic $0$, we have $\underline{\Pic}^{\circ}_{X/K}=\underline{\Pic}^{\circ}_{X/K,\red}$ by Cartier's theorem, but we do not need this fact.)
    The Kummer exact sequence together with divisibility of $\underline{\Pic}^{\circ}_{X/K,\red}(K)$ gives
    \[
    0 \to \NS(X)/\ell^{n} \to H^{2}_{\et}(X,\mu_{\ell^{n}}) \to \Br(X)[\ell^{n}] \to 0,
    \]
    for any prime $\ell$. 
    This is an exact sequence of finite abelian groups (\cite[Theorem 9.10]{OlssonSheavesOnArtinStacks}), so that passing to the limit we get an exact sequence
    \[
    0 \to \NS(X)\otimes_{\ZZ} \ZZ_{\ell} \to H^{2}_{\et}(X,\ZZ_{\ell}(1)) \to T_{\ell}(\Br(X)) \to 0.
    \]
    As $T_{\ell}(\Br(X))=\Hom(\Q_{\ell}/\ZZ_{\ell},\Br(X))$ this is a finite free $\ZZ_{\ell}$-module, so it follows that $T_{\ell}(\Br(X))=\ZZ_{\ell}^{b_{2}-\rho}.$
    Looking now at the Kummer exact sequence for $H^{3}$ and using that $\Br(X)=H^{2}(X,\mathbb{G}_{m})$ is torsion (\cite[Proposition 2.5 (iii)]{AntieauMeierBrauerElliptic}) and running over all $\ell$, we obtain an exact sequence
    \[
    0 \to (\Q/\ZZ)^{b_{2}-\rho} \to \Br(X) \to \bigoplus_{\ell} H^{3}_{\et}(X,\ZZ_{\ell}(1))_{\tors} \to 0.
    \]
    Using independence of choice of algebraically closed base field (use \cite[Tag 07BV]{stacks} and smooth base change \cite[Proposition 2.12]{ZhengEtCohoDM}) and étale-Betti comparison (see e.g.\ \cite[Proposition 4.1.6]{KubrakPrikhodkopHodgeStacks}), we obtain that
    $
    \bigoplus_{\ell} H^{3}_{\et}(X,\ZZ_{\ell}(1))_{\tors}\simeq H^{3}(X_{\mathbb{C}}^{\an},\ZZ)_{\tors}.
    $
\end{proof}

The next result is in the setting of schemes a classical result due to Abrashkin and Fontaine (\cite{Abrashkin90}, \cite{FontainePropreLisseZ}). 

\begin{thm}\label{thm: FontaineAbrashkinDMStacks}
    Let $X\rightarrow \Spec(\ZZ)$ be a proper and smooth DM stack. Then 
    $$
    H^{i}(X_{\mathbb{C}},\Omega^{j}_{X_{\mathbb{C}}})=0
    $$
    for all $i+j\leq 3,$ $i\neq j.$
\end{thm}

\begin{proof}
    The argument here just uses extensions of $p$-adic Hodge theory to stacks and then copy-pastes the argument of Abrashkin and Fontaine.
    More precisely, let $V:=H^{N}_{\et}(X_{\overline{\mathbb{Q}}},\mathbb{Q}_{p}).$
    Assume that there is a $\Qp[\Gal(\overline{\Q}/\Q)]$-filtration
    $$
    0=V_{N+1}\subset V_{N} \subset \cdots \subset V_{1} \subset V_{0}=V
    $$
    such that for all $0\leq i \leq N,$
    we have $V_{i}/V_{i+1}\simeq \mathbb{Q}_{p}(i)^{s_{i}}$ for some $s_{i}\geq 0$.
    Then, by using the de Rham comparison for DM stacks \cite{deRhamCompDMstacks} (this can also be deduced from \cite[Theorem 4.3.25]{KubrakPrikhodkopHodgeStacks}, c.f.\ \textit{loc.cit.}, beginning of page 9) and Proposition \ref{prop: Riemann Hypothesis Weil DM stacks},
    we can repeat the arguments in \cite[page 4]{FontainePropreLisseZ} and \cite[Proposition 5.3]{Abrashkin} to conclude that
    $$
    H^{i}(X_{\mathbb{C}},\Omega^{j}_{X_{\mathbb{C}}/\mathbb{C}})=0
    $$
    for $i+j\leq N,$ $i\neq j.$
    It suffices therefore to find such a filtration. For this, we can just cite Fontaine {\cite[Proposition 1]{FontainePropreLisseZ}}.
\end{proof}

\begin{remark}
    Theorem \ref{thm: FontaineAbrashkinDMStacks} implies that
    $$
        H^{1}(X_{\mathbb{C}},\mathbb{Q})=H^{3}(X_{\mathbb{C}},\mathbb{Q})=0
    $$
    and 
    $$
        c_{1}\colon \Pic(X_{\mathbb{C}})\otimes_{\ZZ}\mathbb{Q}\simeq H^{2}(X_{\mathbb{C}},\mathbb{Q}).
    $$
    When applied to $X=\overline{\mathcal{M}}_{g,n}$ this gives a $p$-adic Hodge theory approach to (some of) the classical vanishing results in \cite{Arbarello-CornalbaCohomology} (with the exception of the degree $5$ cohomology).
\end{remark}

\section{General results on Brauer groups of proper and smooth Deligne-Mumford stacks over the integers}\label{section: General results on Brauer groups}

\subsection{Finiteness of the Brauer group}

In this subsection, we shall prove that
the Brauer groups
$
\Br(X)
$
and
$
\Br(X_{\overline{\QQ}})
$
are finite for a proper and smooth DM stack $X$ over $\Spec(\ZZ)$; see Theorem \ref{Theorem:finiteness of Brauer group}.

For this, we will need to use the relative Picard functor.
Let $S$ be a noetherian scheme and
let
$f\colon X\rightarrow S$
be a proper DM stack over $S$.
We assume that $f$ is flat and $\mathcal{O}_{S}\simeq f_{*}\mathcal{O}_{X}$ holds universally.
We denote by $\underline{\Pic}_{X/S}$ the relative Picard functor, i.e.\ the fppf sheafification of the presheaf that sends a (quasi-separated) scheme $T$ over $S$ to
$$
    \Pic(X\times_{S}T)/\Image(\Pic(T)\rightarrow \Pic(X\times_{S}T)).
$$
Then this is representable by a commutative group algebraic space locally of finite type over $S$.
Moreover $\underline{\Pic}_{X/S}$ commutes with base change in $S$.
We note that by \cite[Theorem 2.2.5]{BrochardPicardFunctorStack} we could also work with the étale sheafification.
See \cite[Section 3]{fringuelli2023picardgroupschememoduli}, \cite{BrochardPicardFunctorStack}, \cite{BrochardPicardStackFiniteness} for background that is needed in our work.

\begin{remark}\label{Remark: geometricallt connected fibers imply cohomological flatness}
In fact, we will mostly consider morphisms $f\colon X\rightarrow \Spec(\ZZ)$ which will be proper and smooth.
    For such morphisms,
    $f_{*}\mathcal{O}_{X}$ is locally free on $\Spec(\ZZ)$ and its formation commutes with base change along any morphism $T \to \Spec(\ZZ)$.
    Indeed, it is well-known that the function sending a point $s \in \Spec(\ZZ)$ to the number of connected components of the geometric fiber of $X$ at $s$ is constant (c.f.\ \cite[Theorem 4.17]{DeligneMumfordModuliCurves}).
    Then, the assertion follows by the same argument as in the proof of \cite[Chapter II, Section 5, Corollary 2]{MumfordAbelian}: More precisely, we choose an \'etale surjection $U \to X$
    from an affine scheme $U$.
    Let
    \[
    0 \to C^{0} \to C^{1} \to C^{2} \to \cdots
    \]
    be the \v{C}ech complex associated to $\mathcal{O}_{X}$ and the \v{C}ech nerve of $U \to X$, so that all entries $C^{i}$ are torsion free $\ZZ$-modules.
    (We note that, since the diagonal
    $X \to X \times_{\Spec(\ZZ)} X$ is a finite morphism,
    the \v{C}ech nerve of $U \to X$ consists of affine schemes.
    However, this complex may not be bounded above in general.)
    Now let us look at the bounded complex $D^{\bullet}$ given by
\[
0 \to C^{0} \to C^{1} \to \Ker(C^{2}\rightarrow C^{3}) \to 0.
\]
Then, for any ring $R,$ we have that
$
H^{0}(X_{R},\mathcal{O}_{X_{R}})=H^{0}(D^{\bullet}\otimes_{\ZZ}R).
$
Since $\Ker(C^{2}\rightarrow C^{3})$ is torsion free, it is flat over $\ZZ$.
Since $H^{i}(X,\mathcal{O}_{X})$ are finitely generated $\ZZ$-modules,
we can use \cite[Chapter II, Section 5, Lemma 1]{MumfordAbelian}
to obtain a bounded complex of finite free $\ZZ$-modules $K^{\bullet}$ with a quasi-isomorphism
of complexes $K^{\bullet}\rightarrow D^{\bullet}.$
We can then follow the argument in \cite[Chapter II, Section 5, Corollary 2]{MumfordAbelian}, to deduce the desired claim.
\end{remark}

Let us record the following consequence of Remark \ref{Remark: geometricallt connected fibers imply cohomological flatness}.

\begin{lem}\label{Lemma:connectedness}
    For a proper and smooth DM stack $X \to \Spec(\ZZ)$, the following assertions are equivalent:
\begin{enumerate}
    \item The geometric fibers of $X$ are connected.
    \item We have $\mathcal{O}_{\Spec(\ZZ)}\simeq f_{*}\mathcal{O}_{X}$ universally.
    \item $X$ is connected.
\end{enumerate}
\end{lem}

\begin{proof}
    $(1) \Leftrightarrow (2)$ follows immediately from Remark \ref{Remark: geometricallt connected fibers imply cohomological flatness}.
    $(1) \Rightarrow (3)$ is easy.
    It remains to show that (3) implies (2).
    By Remark \ref{Remark: geometricallt connected fibers imply cohomological flatness}, it suffices to show that $\ZZ=H^{0}(X,\mathcal{O}_{X})$.
    It is easy to see that
    $H^{0}(X,\mathcal{O}_{X})$
    is finite and flat over $\ZZ$.
    Then, Remark \ref{Remark: geometricallt connected fibers imply cohomological flatness} implies that $H^{0}(X,\mathcal{O}_{X})$ is finite \'etale over $\ZZ$.
    Therefore, it follows from Minkowski's theorem that $\ZZ=H^{0}(X,\mathcal{O}_{X})$ when $X$ is connected.
\end{proof}

Let us now explain the following result:

\begin{thm}\label{Theorem:finiteness of Brauer group}
    Let $f\colon X\rightarrow \Spec(\ZZ)$ be a proper and smooth
    DM stack.
    Then
    $
    \Br(X)
    $
    and
    $
    \Br(X_{\overline{\QQ}})
    $
    are finite.
\end{thm}

\begin{proof}
Since we do not need this result in the following, we will only sketch the proof.
First observe that
$b_{2}(X_{\overline{\QQ}})=\rho(X_{\overline{\QQ}})$ by Theorem \ref{thm: FontaineAbrashkinDMStacks}.
Thus Lemma \ref{lem: short exact sequence with H3} implies that
$
\Br(X_{\overline{\QQ}})
$
is finite.
For the finiteness of
$
\Br(X)
$,
we may assume that $X$ is connected.
Then, by Lemma \ref{Lemma:connectedness}, we have $\mathcal{O}_{\Spec(\ZZ)}\simeq f_{*}\mathcal{O}_{X}$ universally, so that we can use the relative Picard space $\underline{\Pic}_{X/\ZZ}$ in the following.

Since $\Br(X_{\overline{\Q}})$ is finite,
it follows that $\Br(X_{\overline{\ZZ}})$ is also finite.
    We therefore have to see that
    $$
    K_{\ZZ}=\Ker(\Br(X)\rightarrow \Br(X_{\overline{\ZZ}}))
    $$
    is finite. 
    It follows from Theorem \ref{thm: FontaineAbrashkinDMStacks} that
    there exists a dense open subset $U$ of $\Spec(\ZZ)$ such that for all
    $s\in U,$
    $$
    H^{i}(X_{s},\mathcal{O}_{X_{s}})=0
    $$
    for $i=1,2,3$ (where $X_s$ is the fiber of $X$ at $s$).
    By \cite[Proposition 3.2, Theorem 3.6]{fringuelli2023picardgroupschememoduli},
    the relative Picard space
    $\underline{\Pic}_{X_{U}/U}$
    is \'etale over $U$.
    We have an exact sequence of abelian sheaves
    for the fppf topology on $U$:
    \[
    0 \to \underline{\Pic}^{\tau}_{X_{U}/U} \to \underline{\Pic}_{X_{U}/U} \to \text{Num}_{X_{U}/U} \to 0.
    \]
    By (\cite[Theorem 3.6 (iv) (b)]{fringuelli2023picardgroupschememoduli}) $\underline{\Pic}^{\tau}_{X_{U}/U}$ is proper over $U$, and since it is open
    in $\underline{\Pic}_{X_{U}/U}$ \cite[Theorem 3.6 (i)]{fringuelli2023picardgroupschememoduli}, it is finite étale.
    It follows from \cite[Theorem 3.1]{MilneDuality} that
    $H^{i}_{\text{fl}}(U,\underline{\Pic}^{\tau}_{X_{U}/U})$ is finite.
    Furthermore, $\text{Num}_{X_{U}/U}$ is a locally constant étale sheaf of finite free $\ZZ$-modules (c.f.\ the proof of Lemma \ref{Lemma:constantness of Picard scheme}).
    It is then not too hard to see that $H^{1}_{\text{fl}}(U,\text{Num}_{X_{U}/U})$ is finite.
    We deduce then that also $H^{1}_{\text{fl}}(U,\underline{\Pic}_{X_{U}/U})$ is finite.
    Now we look at the 7 term exact sequence (see e.g.\ \cite[Appendix B, page 309]{MilneEtCoho}) for the Leray spectral sequence in the fppf topology for $f_{U}\colon X_{U}\rightarrow U$ and the sheaf $\mathbb{G}_{m}$ and see that
    $$
    \Ker(\Br(X_{U})\rightarrow \Br(X_{\overline{U}}))/\Image(\Br(U))
    $$
    is finite. Here $\overline{U}=U\times_{\Spec(\ZZ)}\Spec(\overline{\ZZ}).$ It then suffices to see that $K_{\ZZ}\cap \Image(\Br(U))$ (intersection in $\Br(X_{U})$) is finite. This is not too hard to see using that $X$ acquires a $\mathcal{O}_{K}$-valued point, where $K/\mathbb{Q}$ is some finite Galois extension, and that $\Br(\mathcal{O}_{K})$ is finite (\cite[Proposition 2.4]{GrothendieckGroupedeBrauerIII}).
\end{proof}

\subsection{Deducing vanishing over the integers}\label{Subsection:Deducing vanishing over the integers}

The main result of this section is the following:
\begin{prop}\label{Proposition: deducing vanishing over Z}
    Let $f\colon X\rightarrow \Spec(\ZZ)$ be a proper and smooth DM stack with geometrically connected fibers.
    Assume that either of the following conditions holds:
    \begin{enumerate}
        \item $X(\ZZ)\neq \emptyset$ and $X_{\overline{\Q}}$ is simply connected.
        \item $\underline{\Pic}_{X/\ZZ}$ is a constant sheaf associated with a finitely generated $\ZZ$-module.
    \end{enumerate}
    Then the natural homomorphism
    \[
    \Br(X) \to \Br(X_{\overline{\Q}})
    \]
    is injective.
\end{prop}

Before we start the proof of Proposition \ref{Proposition: deducing vanishing over Z}, we need two lemmas.
Let us consider $f\colon X\rightarrow \Spec(\ZZ),$ a proper and smooth DM stack with geometrically connected fibers.

\begin{lem}\label{lem Vanishing Galois coho Pic}
Assume that $X_{\Qbar}$ is simply connected.
Then we have
$$
H^{1}_{\et}(\Spec(\Q),\underline{\Pic}_{X_{\Q}/\Q})=0.
$$
\end{lem}
\begin{proof}
Since $H^{1}(X_{\Q},\mathcal{O}_{X_{\Q}})=0$ by Theorem \ref{thm: FontaineAbrashkinDMStacks}, we have that $\underline{\Pic}^{0}_{X_{\Q}/\Q}(\Qbar)=0$ (use \cite[Théorème 4.1.3]{BrochardPicardFunctorStack}), so that 
$
\Pic(X_{\Qbar})=\NS(X_{\Qbar}).
$
Thus $\Pic(X_{\Qbar})$ is a finitely generated abelian group by \cite[Theorem 3.4.1]{BrochardPicardStackFiniteness}.
Since $X_{\Qbar}$ is simply connected, we see by the Kummer exact sequence that $\Pic(X_{\Qbar})$ is torsion free.
Now observe that $\Pic(X_{\Qbar})$ is an everywhere unramified representation of $\Gal(\Qbar/\Q).$
Namely, let $\ell$ be a prime and we want to show that $\Pic(X_{\overline{\Q}})$ is unramified at $\ell.$
Let $p\neq \ell$ be an auxiliary prime.
By the Kummer exact sequence, we have a $\Gal(\overline{\Q}/\Q)$-equivariant embedding
$$
\Pic(X_{\overline{\Q}})\otimes_{\ZZ}\Qp \hookrightarrow H^{2}_{\et}(X_{\overline{\Q}},\Qp(1)).
$$
Since we know that $\Pic(X_{\overline{\Q}})$ is torsion free,
Lemma \ref{lem: Galois rep is everywhere unramified and crystalline at p} implies that it is unramified at $\ell$, as desired.

The action of $\Gal(\overline{\Q}/\Q)$ on $\Pic(X_{\overline{\Q}})$ is then trivial by Minkowski's theorem.
\begin{comment}
    the Galois extension $\overline{\Q}^{\Ker(\rho)}$ of $\Q$, where $\rho$ is the Galois representation $\Pic(X_{\overline{\Q}}),$ is everywhere unramified.
\end{comment}
In total, it follows that
$$
    H^{1}_{\Gal}(\Gal(\Qbar/\Q),\Pic(X_{\Qbar}))=\Hom_{\text{cts}}(\Gal(\Qbar/\Q),\Pic(X_{\Qbar}))=0,
$$
since $\Pic(X_{\Qbar})$ is torsion free.
\end{proof}
\begin{lem}\label{Lemma description Kernel}
Assume that $X_{\Qbar}$ is simply connected.
Then
\[
\Ker(\Br(X_{\Q})\rightarrow \Br(X_{\Qbar}))=\Image(\Br(\Q)\rightarrow \Br(X_{\Q})).
\]
\end{lem}
\begin{proof}
Since $R^{1}f_{*}\mathbb{G}_{m}=\underline{\Pic}_{X/\ZZ}$
by \cite[Proposition 2.1.3]{BrochardPicardFunctorStack},
we have
$$
H^{1}_{\et}(\Spec(\Q),R^{1}f_{*}\mathbb{G}_{m})=H^{1}_{\et}(\Spec(\Q),\underline{\Pic}_{X_{\Q}/\Q})=0
$$
by Lemma \ref{lem Vanishing Galois coho Pic}.
By the 7 term exact sequence for the Leray spectral sequence in étale cohomology for $f_{\Q}\colon X_{\Q}\rightarrow \Spec(\Q)$ and $\mathbb{G}_{m}$ (see \cite[Théorème A.2.8.]{BrochardPicardFunctorStack}),
we then have a surjection
\[
\Br(\Q) \to \Ker(\Br(X_{\Q})\rightarrow H^{0}_{\et}(\Spec(\Q),R^{2}f_{*}\mathbb{G}_{m})).
\]
(Here we also use $\Br(Y)=H^{2}_{\et}(Y,\mathbb{G}_{m})$ for a noetherian and regular DM stack $Y$.)
By the sheaf property of
$R^{2}f_{*}\mathbb{G}_{m}$, we see that
$H^{0}(\Spec(\Q),R^{2}f_{*}\mathbb{G}_{m})$
is a subgroup of
$\colim_{K/\Q}H^{0}(\Spec(K),R^{2}f_{*}\mathbb{G}_{m}),$
where the colimit is indexed by finite extension of $\Q$ contained in $\overline{\Q}.$
This implies that
\[
\Ker(\Br(X_{\Q})\rightarrow \Br(X_{\Qbar})) \subset \Ker(\Br(X_{\Q})\rightarrow H^{0}_{\et}(\Spec(\Q),R^{2}f_{*}\mathbb{G}_{m})).
\]
Now observe that
$\Image(\Br(\Q)\rightarrow \Br(X_{\Q}))$
lies in
$\Ker(\Br(X_{\Q})\rightarrow \Br(X_{\Qbar}))$.
\end{proof}

Now we can prove Proposition \ref{Proposition: deducing vanishing over Z}:

\begin{proof}[Proof of Proposition \ref{Proposition: deducing vanishing over Z}]
We first put ourselves in the situation (1).
Consider a Brauer class
$\alpha\in \Br(X)$
such that $\alpha\vert_{X_{\Qbar}}=0.$
By Lemma \ref{Lemma description Kernel},
we have 
$$
\alpha\vert_{X_{\Q}}=f_{\Q}^{*}(\beta)
$$
for some $\beta\in \Br(\Q)$.
We choose a section $s\colon \Spec(\ZZ)\rightarrow X$ of the structure map.
Passing to Brauer groups, we get the following commutative diagram
$$
\xymatrix{
\Br(\Q) \ar[r]^-{f^*_{\QQ}}& \Br(X_{\Q}) \ar[r]^-{s^*_{\QQ}} &  \Br(\Q) \\
\Br(\ZZ) \ar[u] \ar[r]^-{f^*}& \Br(X) \ar[u] \ar[r]^-{s^*} & \Br(\ZZ) \ar[u],
}
$$
where both horizontal rows compose to the identity.
We have that $\Br(\ZZ)=0$ by class field theory.
This implies that $\beta=0$, so that $\alpha\vert_{X_{\Q}}=0$.
This by the injectivity
(see \cite[Proposition 2.5 (iv)]{AntieauMeierBrauerElliptic}) of
$\Br(X)\rightarrow \Br(X_{\Q})$
implies that $\alpha=0,$ as desired.

Next, we put ourselves in the situation (2).
We look at again
the 7 term exact sequence for the Leray spectral sequence in the fppf topology for $f\colon X\rightarrow \Spec(\ZZ)$ and the sheaf $\mathbb{G}_{m}$.
By the same argument as in Lemma \ref{Lemma description Kernel} using that $\Br(\ZZ)=0$ and 
$$
H^{1}_{\text{fl}}(\Spec(\ZZ),\underline{\Pic}_{X/\ZZ})=H^{1}_{\et}(\Spec(\ZZ),\underline{\Pic}_{X/\ZZ})=0
$$
(\cite[Chapter III, Theorem 3.9]{MilneEtCoho}),
we see that
$
\Br(X)\rightarrow \Br(X_{\overline{\ZZ}})
$
is injective.
Now we use the injectivity of
$\Br(X_{\overline{\ZZ}}) \to \Br(X_{\overline{\Q}})$.
\end{proof}

\begin{cor}\label{Corollary: Brauer group of moduli spaces of curves over Z}
    For integers $g, n \geq 0$ with $2g-2+n > 0$,
    the condition (1) in Proposition \ref{Proposition: deducing vanishing over Z} is satisfied for $\overline{\mathcal{M}}_{g,n,\ZZ}.$
    In particular, if $\Br(\overline{\mathcal{M}}_{g, n, \overline{\Q}})=0$, then we have
    $\Br(\overline{\mathcal{M}}_{g,n,\ZZ})=0$.
\end{cor}

\begin{proof}
    By \cite[Proposition 1.1]{BoggiPikaartGaloisCovers} (c.f.\ \cite[Theorem 5.8]{fringuelli2023picardgroupschememoduli}),
    we see that
    $\overline{\mathcal{M}}_{g,n,\overline{\Q}}$
    is simply connected.
    
    We shall show that
    $\overline{\mathcal{M}}_{g,n}(\ZZ)\neq \emptyset$.
    We first remark that if
    $\overline{\mathcal{M}}_{g,n}(\ZZ)\neq \emptyset$, then $\overline{\mathcal{M}}_{g,n+1}(\ZZ)\neq \emptyset$
    since the contraction morphisms
    $\overline{\mathcal{M}}_{g,n+1} \to \overline{\mathcal{M}}_{g,n}$
    (see \cite[Section 2]{Knudsen}) have a section
    $\overline{\mathcal{M}}_{g,n} \to \overline{\mathcal{M}}_{g,n+1}$.
    
    To show that $\overline{\mathcal{M}}_{g,n}(\ZZ)\neq \emptyset$, we proceed by induction on $g$.
    We assume that $g=0$.
    If $n=3$, the projective line
    $\PP^1_{\ZZ} \to \Spec (\ZZ)$ with the three sections
    defined by $0$, $1$, $\infty$ gives an element of $\overline{\mathcal{M}}_{0,3}(\ZZ)$.
    By the remark in the previous paragraph, we then have 
    $\overline{\mathcal{M}}_{0,n}(\ZZ)\neq \emptyset$ for every $n \geq 3$.
    We assume that the assertion is true for some $g \geq 0$.
    By using the clutching morphisms
    $
    \overline{\mathcal{M}}_{g,n+2} \to \overline{\mathcal{M}}_{g+1,n}
    $
    (see \cite[Section 3]{Knudsen}), and the assumption that $\overline{\mathcal{M}}_{g,n+2}(\ZZ) \neq \emptyset$,
    we see that $\overline{\mathcal{M}}_{g+1,n}(\ZZ)\neq \emptyset$ for any $n$ with $2(g+1)-2+n > 0$.
    This concludes the proof.
\end{proof}

\begin{remark}\label{Remark:Constancy of Picard scheme of moduli space of stable curves}
    By the main results of \cite{fringuelli2023picardgroupschememoduli}, the condition (2) of Proposition \ref{Proposition: deducing vanishing over Z} is also satisfied for $\overline{\mathcal{M}}_{g,n,\ZZ}$ for $g\leq 5.$
\end{remark}

The following statement, together with Remark \ref{Remark:Constancy of Picard scheme of moduli space of stable curves}, will be needed in the proof of Theorem \ref{thm: main result} for genus $3$ (see Theorem \ref{Theorem: g=3 in characteristic 0}).

\begin{lem}\label{Lemma: vanishing of Brauer class and reduction}
Assume that we are in the situation (2) in Proposition \ref{Proposition: deducing vanishing over Z}.
Let 
$\alpha\in \Br(X_{\overline{\QQ}})$
be an element of order $n$.
Assume that there is a prime $p$ such that
$n$ is not divisible by $p$ and the image
$\alpha\vert_{X_{\overline{\mathbb{F}}_{p}}}$
of $\alpha$ under the specialization map
$\Br(X_{\overline{\QQ}}) \to \Br(X_{\overline{\mathbb{F}}_{p}})$
is zero.
Then we have $\alpha=0.$
\end{lem}

\begin{proof}
We consider the following diagram comparing the Kummer sequences
$$
\xymatrix{
0 \ar[r] & \Pic(X_{\overline{\mathbb{Q}}})\otimes_{\ZZ}\ZZ/n\ZZ \ar[r] \ar[d] & H^{2}(X_{\overline{\mathbb{Q}}},\mu_{n}) \ar[r] \ar[d] & \Br(X_{\overline{\mathbb{Q}}})[n] \ar[d] \ar[r] & 0 \\
0 \ar[r] & \Pic(X_{\overline{\mathbb{F}}_{p}})\otimes_{\ZZ}\ZZ/n\ZZ \ar[r] & H^{2}(X_{\overline{\mathbb{F}}_{p}},\mu_{n}) \ar[r] & \Br(X_{\overline{\mathbb{F}}_{p}})[n] \ar[r] & 0.}
$$
The left downward arrow is an isomorphism by (2), the middle downward arrow is an isomorphism by base change (see the proof of \cite[Proposition 7.2]{fringuelli2023picardgroupschememoduli}).
It follows that the right downward arrow is also an isomorphism.
Therefore, our assumption implies that $\alpha=0.$
\end{proof}

\subsection{Examples}\label{Subsection: Examples}

As a supplement to Proposition \ref{Proposition: deducing vanishing over Z} (and our main result), we give a few examples here.
The following claim gives a sufficient criterion for checking that $\underline{\Pic}_{X/\ZZ}$ is a constant sheaf.

\begin{lem}\label{Lemma:constantness of Picard scheme}
Let $f\colon X\rightarrow \Spec(\ZZ)$ be a proper and smooth DM stack with geometrically connected fibers.
We assume that
\[
H^1(X_{s}, \mathcal{O}_{X_{s}})=H^2(X_{s}, \mathcal{O}_{X_{s}})=0
\]
for all points $s\in \Spec(\ZZ)$.
Then $\underline{\Pic}_{X/\ZZ}$ is constant on a finitely generated abelian group.
\end{lem}

\begin{proof}
Once one knows that the sheaf $\underline{\Pic}_{X/\ZZ}$ is constant, it follows that the group must be finitely generated (c.f.\ the proof of Lemma \ref{lem Vanishing Galois coho Pic}).
Consider $\underline{\Pic}^{\tau}_{X/\ZZ}\subset \underline{\Pic}_{X/\ZZ}$ (see \cite[Section 3]{fringuelli2023picardgroupschememoduli}).
By our assumptions, we know the following properties:
\begin{itemize}
    \item $\underline{\Pic}_{X/\ZZ}$ is \'etale over $\Spec(\ZZ)$.
    \item $\underline{\Pic}^{\tau}_{X/\ZZ}\subset \underline{\Pic}_{X/\ZZ}$ is open and closed.
    \item $\underline{\Pic}^{\tau}_{X/\ZZ}$ is proper over $\Spec(\ZZ)$.
\end{itemize}
See \cite[Proposition 3.2, Theorem 3.6]{fringuelli2023picardgroupschememoduli}.
In particular, we know that $\underline{\Pic}^{\tau}_{X/\ZZ}$ is finite étale over $\Spec(\ZZ),$ so that it has to be constant by Minkowski's theorem.
We look at the exact sequence
\[
0 \to \underline{\Pic}_{X/\ZZ}^{\tau} \to \underline{\Pic}_{X/\ZZ} \to \Num_{X/\ZZ} \to 0.
\]
(For any geometric point $s \rightarrow \Spec(\ZZ),$
the fiber
$
\Num_{X/\ZZ}(s)
$
is the 
maximal torsion free quotient of
$\NS(X_{s})$.)
Now it suffices to show that $\Num_{X/\ZZ}$ is constant.
By the same argument as \cite[Proposition 4.2]{ekedahl2012moduliperiodssimplyconnected}, we see that $\Num_{X/\ZZ}$ is locally constant, and hence constant by Minkowski's theorem.
(In \textit{loc.cit.}, the authors assume the projectivity of the morphism $f$, but the properness is enough for the argument.)
\end{proof}

We give an example where $X(\ZZ)=\emptyset,$ but the condition (2)
of Proposition \ref{Proposition: deducing vanishing over Z} still holds, which explains why we put this condition.

\begin{examp}\label{examp: flag}
    Let $G$ be a non quasi-split reductive group scheme over $\Spec(\ZZ).$
    Consider the scheme $X$ of Borel subgroups of $G.$ This is a projective and smooth scheme over $\Spec(\ZZ).$
    By assumption, we have that $X(\ZZ)=\emptyset.$
    However we still have that
    $\Br(X)=\Br(X_{\overline{\QQ}})=0.$
    In fact, since $X_{\overline{\QQ}}$ is rational, we have $\Br(X_{\overline{\QQ}})=0$.
    By Kempf's theorem, we have
    \[
    H^i(X_{\Fp}, \mathcal{O}_{X_{\Fp}})=0
    \]
    for any prime $p >0$ and $i > 0$, which implies that $\underline{\Pic}_{X/\ZZ}$ is constant on a finitely generated abelian group by Lemma \ref{Lemma:constantness of Picard scheme}.
    This then implies via Proposition \ref{Proposition: deducing vanishing over Z} that $\Br(X)=0,$ as desired.
\end{examp}

We next give an easy example
of a proper and smooth DM stack $X$ over $\Spec(\ZZ)$
such that $\Br(X) \to \Br(X_{\overline{\Q}})$ is not injective, in which the conditions (1) and (2) in Proposition \ref{Proposition: deducing vanishing over Z} are not satisfied.
We note that for a finite group $G$, the classifying stack
$[\Spec(\ZZ)/G]$
is proper and \'etale over $\Spec(\ZZ)$.

\begin{examp}\label{examp: Brauer over integers not zero}
    Consider $X=[\Spec(\ZZ)/G]$
    for $G:=\ZZ/2\ZZ$.
    Here both (1) and (2) of Proposition \ref{Proposition: deducing vanishing over Z} fail.
    Nevertheless, since $\Pic(\ZZ)=0$ and $\Br(\ZZ)=0,$ we obtain that
    $$
    \Br(X)=H^{2}(G, \mathbb{G}_{m}(\ZZ)) \simeq \ZZ/2\ZZ
    $$
    (here the action of $G$ on $\mathbb{G}_{m}(\ZZ)=\lbrace \pm 1 \rbrace$ is the trivial one).
    The same argument also implies that
    $\Br(X_{\overline{\Q}})= H^{2}(G, \mathbb{G}_{m}(\overline{\Q}))= 0.$
\end{examp}

Finally, let us include the following example,
which shows our vanishing result (Theorem \ref{thm: main result}) does not hold for proper and smooth DM stacks over $\Spec (\ZZ)$ in general.

\begin{examp}\label{examp: Brauer over geometric generic fiber not zero}
    Consider $X=[\Spec(\ZZ)/S_{4}],$ where $S_{4}$ is the symmetric
    group on $4$ elements.
    By the same argument as in Example \ref{examp: Brauer over integers not zero},
    we have
    \[
    \Br(X)=H^{2}(S_4, \mathbb{G}_{m}(\ZZ)) \simeq \ZZ/2\ZZ, \quad \Br(X_{\overline{\Q}})= H^{2}(S_4, \mathbb{G}_{m}(\overline{\Q})) \simeq \ZZ/2\ZZ.
    \]
\end{examp}

\section{Brauer groups of moduli stacks of stable curves}\label{Section: Brauer groups of moduli spaces of stable curves}

In this section,
we prove our vanishing results when working over $\overline{\Q}.$
Recall that the case where $g=0$ is easy to prove; see Remark \ref{Remark:genus zero case}.
So it is enough to consider the case where $g \geq 1$.

\subsection{Genus one}

The next proposition is easy to prove because one has a very good handle on the geometry of $\mathcal{M}_{1,1,\mathbb{C}}$.

\begin{prop}\label{Prop: Vanishing (1,1) bar Q}
$\Br(\mathcal{M}_{1,1,\overline{\Q}})=0;$ in particular, we have that
$\Br(\overline{\mathcal{M}}_{1,1,\overline{\Q}})=0.$
\end{prop}
\begin{proof}
This is proved in \cite[Theorem 1.1 (1)]{AntieauMeierBrauerElliptic}.
Here we give an independent argument:
We first claim that 
\[
\Br(\mathcal{M}_{1,1,\overline{\Q}})\simeq\Br(\mathcal{M}_{1,1,\mathbb{C}}).
\]
In fact, this can be deduced by comparing the Kummer exact sequences for $\mathcal{M}_{1,1,\overline{\Q}}$ and $\mathcal{M}_{1,1,\mathbb{C}},$ using invariance of étale cohomology with torsion coefficients under choice of algebraically closed base field (c.f.\ proof of Lemma \ref{lem: short exact sequence with H3}) and $\Pic(\mathcal{M}_{1,1,\overline{\Q}})\simeq \Pic(\mathcal{M}_{1,1,\mathbb{C}})\simeq \mathbb{Z}/12\mathbb{Z}$ \cite{FultonOlssonPicM11}.
Repeating this argument, but now using as an input the string of identifications
$\Pic(\mathcal{M}^{\an}_{1,1,\mathbb{C}})\simeq \Pic(\mathcal{M}_{1,1,\mathbb{C}})\simeq \mathbb{Z}/12\ZZ$ (see \cite{ArbarelloCornalbaPicCurves})
and étale-Betti comparison for DM stacks, we see that
$$
\Br(\mathcal{M}^{\an}_{1,1,\mathbb{C}})=\Br(\mathcal{M}_{1,1,\mathbb{C}}).
$$
Here
$\Br(\mathcal{M}^{\an}_{1,1,\mathbb{C}}):=H^2(\mathcal{M}^{\an}_{1,1,\mathbb{C}}, \mathbb{G}_m)_{\tors}$.

Thus it suffices to show that $\Br(\mathcal{M}^{\an}_{1,1,\mathbb{C}})=0$.
Observe that it then suffices to show that
$$
\Br(\mathcal{M}^{\an}_{1,1,\mathbb{C}})\simeq H^{3}(\mathcal{M}^{\an}_{1,1,\mathbb{C}},\mathbb{Z})_{\tors}=0.
$$
But
$$
H^{3}(\mathcal{M}^{\an}_{1,1,\mathbb{C}},\mathbb{Z})=H^{3}(\SL_{2}(\ZZ),\ZZ)=0
$$
using that $\SL_{2}(\ZZ)=\ZZ/4\ZZ \ast_{\ZZ/2\ZZ} \ZZ/6\ZZ$ (use \cite[Corollary 7.7]{BrownCohoGroups}).
\end{proof}
The cases $(g,n)=(1,n)$ for $2\leq n \leq 6$ lie deeper and require very recent work of Battistella-Di Lorenzo \cite{battistella2024wallcrossingintegralchowrings} (who compute the integral Chow ring in the cases $(1,3)$ and $(1,4)$).

\begin{prop}\label{Prop: vanishing g 1 n 2 3 4 5 6}
    $\Br(\overline{\mathcal{M}}_{1,n,\overline{\Q}})=0$
    for $2\leq n \leq 6.$
\end{prop}
\begin{proof}
    If not specified, all stacks in this proof will be over $\overline{\Q}.$
    (In fact, the following argument also gives another proof of the equality $\Br(\overline{\mathcal{M}}_{1,1,\overline{\Q}})=0$.)
    
    By \cite{RSPW19GenusOneLogarithmic}, \cite{Smyth19DiagramG1} (see also the introduction to \cite{battistella2024wallcrossingintegralchowrings}) we have a 
    diagram
    $$
        \xymatrix{
        &    &  \ar[ld]^-{\rho_{\frac{3}{2}}} \overline{\mathcal{M}}_{1,n}(\frac{3}{2}) \ar[rd]^-{\rho_{2}} & & \ar[ld]^-{\rho_{\frac{5}{2}}} \cdots \ar[rd]^-{\rho_{n-2}} & \\     
        \overline{\mathcal{M}}_{1,n} \ar[r]^-{\rho_{1}}   & \overline{\mathcal{M}}_{1,n}(1) &  & \overline{\mathcal{M}}_{1,n}(2) & \cdots & \overline{\mathcal{M}}_{1,n}(n-2) \ar[d]^-{\rho_{n-1}} & \\
        & & & & & \overline{\mathcal{M}}_{1,n}(n-1).
        }
    $$
    We know the following:
    \begin{itemize}
        \item By \cite[Theorem 1.12]{battistella2024wallcrossingintegralchowrings} $m\leq 5,$ the morphisms $\rho_{m}$ are weighted blow-ups in loci which are of codimension greater or equal than $2.$
        \item The morphisms $\rho_{\frac{j}{2}}$ are ordinary blow-ups (this is by construction).
        \item By \cite[Theorem 1.5]{battistella2024wallcrossingintegralchowrings}, the stacks $\overline{\mathcal{M}}_{1,n}(m)$ are smooth for $m\leq 5.$ Observe that this also implies (via Lemma \ref{lem: blow up regular center smooth DM stack}) that the stacks $\overline{\mathcal{M}}_{1,n}(\frac{j}{2})$ are smooth.
        \item For $2\leq n \leq 6,$ we have that $\Br(\overline{\mathcal{M}}_{1,n}(n-1))=0.$
        Namely, for $2\leq n \leq 5,$ this follows from \cite[Proposition 1.8]{battistella2024wallcrossingintegralchowrings} together with \cite[Theorem 1.2]{ShinBrauerWeighted}. 
        For $n=6,$ we use \cite[Corollary 1.5.5, Proposition 1.7.1]{LP19ModularCompactification}, together with
        $$
            \Br(\text{Gr}(k,m))=0.
        $$
        \item By \cite[Remark 3.2.10]{WeightedBlowUps} a weighted blow-up is an isomorphism away from the blow-up center.
    \end{itemize}
    Finally, observe that all entries in the above diagram are irreducible as DM stacks, so that we may apply Corollary \ref{cor: Brauer groups blow ups} step-by-step to deduce the desired vanishing.
\end{proof}

\subsection{Genus two}
Next, we quickly explain why the case $g=2$ is already contained in the literature:
\begin{prop}[Di Lorenzo-Pirisi]\label{Proposition:genus two case}
    The equality
    $
        \Br(\overline{\mathcal{M}}_{2,\overline{\mathbb{Q}}})=0
    $
    holds.
\end{prop}
\begin{proof}
    This result is contained in \cite[Theorem A.1]{DiLorenzo-PirisiHyperEllEven}.
    Although this is probably trivial for the experts, let us explain why.
    Let $\mathcal{H}_{g,\overline{\mathbb{Q}}}$
    be the moduli stack over $\overline{\mathbb{Q}}$ of hyper-elliptic curves of genus $g$
    (in the sense of \cite{ArsieVistoliCyclicCovers}, i.e.\ a genus $g$ hyper-elliptic curve over a scheme $S$ consists of $X\rightarrow P \rightarrow S,$ where
    $X\rightarrow S$ is a proper and smooth morphism with all geometric fibers connected curves of genus $g,$ $X\rightarrow P$ is a morphism of schemes which is finite, fppf and of degree $2$ and $P\rightarrow S$ is a family of genus $0$ curves).
    This is an irreducible smooth algebraic stack over $\overline{\mathbb{Q}}$ of dimension $2g-1,$ as proven in \cite{ArsieVistoliCyclicCovers}.
    For $g\geq 2$, the morphism
    $
        \mathcal{H}_{g,\overline{\mathbb{Q}}}\rightarrow \mathcal{M}_{g,\overline{\mathbb{Q}}}
    $
    that forgets $P$
    is a closed immersion.
    For $g=2,$ it follows that
    $$
        \mathcal{H}_{2,\overline{\mathbb{Q}}}\simeq \mathcal{M}_{2,\overline{\mathbb{Q}}}
    $$
    since $\mathrm{dim} (\mathcal{H}_{2,\overline{\mathbb{Q}}})=\mathrm{dim} (\mathcal{M}_{2,\overline{\mathbb{Q}}})=3$.
    The stack of stable hyper-elliptic genus $g$ curves $\overline{\mathcal{H}}_{g,\overline{\mathbb{Q}}}$ is identified with the closure of $\mathcal{H}_{g,\overline{\mathbb{Q}}}$ inside of $\overline{\mathcal{M}}_{g,\overline{\mathbb{Q}}}.$
    It follows that
    $$
        \overline{\mathcal{H}}_{2,\overline{\mathbb{Q}}}\simeq \overline{\mathcal{M}}_{2,\overline{\mathbb{Q}}}.
    $$
    But the mentioned result of Di Lorenzo-Pirisi implies that
    $
        \Br(\overline{\mathcal{H}}_{2,\overline{\mathbb{Q}}})=0.
    $
\end{proof}

\subsection{Genus three}\label{Subsection: genus three}

The purpose of this subsection is to prove that
$\Br(\overline{\mathcal{M}}_{3, \overline{\QQ}})=0$; see Theorem \ref{Theorem: g=3 in characteristic 0}.
For this, we first recall some results from \cite{Lorenzo-Pirisi}.

It is shown in \cite{Lorenzo-Pirisi} that for a field $k$ of characteristic $p \geq 3$, there is a natural isomorphism
\[
\Br(\mathcal{M}_{3, k}) \simeq \Br(k) \oplus \ZZ/2\ZZ \oplus B_p
\]
where $B_p$ is a $p$-primary torsion group, while for $\text{char}(k)=0,$ we have
\[
\Br(\mathcal{M}_{3,k}) \simeq \Br(k) \oplus \ZZ / 2\ZZ.
\]
We denote the generator of
$\ZZ/2\ZZ \subset \Br(\mathcal{M}_{3, k})$
by $\gamma$.
This generator $\gamma$ is constructed explicitly in \cite{Lorenzo-Pirisi}.
Let $C$ be a projective smooth curve over $k$ of genus $3$.
Let $\gamma(C) \in \Br(k)[2]$ be the element obtained by pulling back $\gamma \in \Br(\mathcal{M}_{3, k})$ along the morphism $\Spec (k) \to \mathcal{M}_{3, k}$ corresponding to $C$.
We shall recall the description of $\gamma(C)$ in a special case which will be needed for our purposes.

%\begin{examp}\label{Example: bitangents and the description of gamma}
    %Let $C \subset \PP^2_k$ be a smooth quartic curve over $k$.
    %It is well known that
    %there are exactly $28$ bitangents of $C_{\overline{k}} \subset \PP^2_{\overline{k}}$.
%\end{examp}

\begin{examp}\label{Example: cubic surface and the description of gamma}
    Let $X \subset \PP^3_k$ be a smooth cubic surface over $k$.
    It is well-known that $X_{\overline{k}}$ contains exactly $27$ lines.
    Since the absolute Galois group of $k$ acts on the set of lines in $X_{\overline{k}}$, we obtain the associated \'etale algebra $E_X$ over $k$ of degree $27$.
    Let $P \in X(k)$ be a $k$-rational point which is not contained in any line in $X_{\overline{k}}$ and let
    $H \subset \PP^3_k$ be a hyperplane such that $P$ is not contained in $H$.
    Then we have a morphism
    \[
    \mathrm{Bl}_P (X) \to H
    \]
    induced by the projection onto $H$.
    This morphism expresses $\mathrm{Bl}_P (X)$ as a double covering of $H$ whose ramification locus is a smooth quartic curve $C \subset H$ over $k$.
    Then
    \[
    \gamma(C) \in \Br(k)[2] \simeq H^2_{\et}(\Spec (k), \mu_2)
    \]
    agrees with the second Galois-Stiefel-Whitney class $\alpha_2(E_X)$
    as explained in \cite[Section 3.4, Section 3.5]{Lorenzo-Pirisi}.
    We recall that $\alpha_2(E_X)$ is the degree $2$ term of the total Galois-Stiefel-Whitney class $\alpha_{\mathrm{tot}}(E_X)$; see \cite[Section 3.3]{Lorenzo-Pirisi} for the definition.
    Strictly speaking, the second Galois-Stiefel-Whitney class $\alpha_2(E_X)$ is an element of $H^2_{\et}(\Spec (k), \mu^{\otimes 2}_2)$.
    However, since $\mu_2$ can be naturally identified with $\ZZ/2\ZZ$, we may regard $\alpha_2(E_X)$ as an element of $H^2_{\et}(\Spec (k), \mu_2)$.
\end{examp}

Let $p$ be a prime such that $p \geq 11$.
    Let $\Fq$ be a finite field with $q=p^r$ elements.
    We consider a local field $F=\Fq((t))$.
In the following Example \ref{Example: nontrivial gamma, q-1/2 is even} (resp.\ Example \ref{Example: nontrivial gamma, q-1/2 is odd}), we will give an example of a smooth quartic curve $C$ over $F$ such that
$\gamma(C) \in \Br(F)[2]$
is not zero when $(q-1)/2$ is even (resp.\ odd).

Let $E_1:=F(\sqrt{t})$, which is a totally ramified extension of $F$ of degree $2$.
    Let $u \in \Fq^\times$ be an element such that $\sqrt{u}$ is not contained in $\Fq$ and let
    $E_2:=F(\sqrt{u})$.
    Then $E_2$ is an unramified extension of $F$ of degree $2$.
    Let $K:=F(\sqrt{t}, \sqrt{u})$.

\begin{examp}\label{Example: nontrivial gamma, q-1/2 is even}
    We set
    \begin{align*}
        &u_1:= 1 + \sqrt{t}, &u_2 &:= 1 - \sqrt{t}, \\
        &u_3:= 1 + \sqrt{u}, &u_4 &:= 1 - \sqrt{u}, \\
        &u_5:= 2 + t, &u_6 &:= t.
    \end{align*}
    Since $u_1$ is conjugate to $u_2$, they come from a closed point
    $U_{12} \colon \Spec (E_1) \hookrightarrow \PP^2_{F}$.
    Similarly, $u_3$ and $u_4$ come from a closed point
    $U_{34} \colon \Spec (E_2) \hookrightarrow \PP^2_{F}$.
    One can check that they satisfy the conditions
    \[
    u_i \neq u_j \, \,  (i \neq j), \quad u_i+u_j+u_k \neq 0 \, \, (\text{$i, j, k$ are distinct}), \quad \sum^6_{i=1} u_i \neq 0, 
    \]
    or equivalently, the $6$ points $U_i:=[u^{-2}_i, u^{-3}_i, 1] \in \PP^2_{K}$ are in a general position (i.e.\ they are not on a conic and no three of $u_i$ lie on a line).
    It thus follows that the blow-up $X$ of $\PP^2_{F}$ along the closed points $U_{12}$, $U_{34}$, $U_5$, $U_6$ is naturally a smooth cubic surface $X \hookrightarrow \PP^3_{F}$ over $F$.
    By the same argument as in \cite[Section 3.4]{Lorenzo-Pirisi}, we can check that
    \[
    E_X = E^4_1 \times E^4_2 \times K \times F^7.
    \]
    We note that $\alpha_{\mathrm{tot}}(F)= 1$.
    We have $\alpha_{\mathrm{tot}}(E_1)= 1 + \{ t  \}$ and $\alpha_{\mathrm{tot}}(E_2)= 1 + \{ u  \}$.
    We have $\alpha_1(K) = 0$ and $\alpha_2(K) = \{ t, u \} + \{ -1, tu \}$ by the same computation as in \cite[Example 3.11]{Lorenzo-Pirisi}.
    Using the multiplicativity of the total Galois-Stiefel-Whitney classes, we obtain $\alpha_2(E_X) = \alpha_2(K) = \{ t, u \} + \{ -1, tu \}$.

    We have $\Br(F)[2] \simeq \ZZ/2\ZZ $ by local class field theory.
    We further identify $\ZZ/2\ZZ$ with $\mu_2(F)= \{-1, 1 \}$.
    Under these identifications, $\{ t, u \}$ agrees with the Hilbert symbol $(t, u)$ on $F$.
    By \cite[Chapter XIV, Section 3, Corollary]{SerreLocalfields}, we then have $\{ t, u \} = u^{(q-1)/2}$.
    Similarly, we get $\{ -1, tu \} = (-1)^{(q-1)/2}$.
    Since $\sqrt{u}$ is not contained in $\Fq$, we see that $\{ t, u \} = u^{(q-1)/2} =-1$.
    It follows that
    \[
    \alpha_2(E_X)= (-1)^{(q-1)/2+1}.
    \]
    Therefore, for a smooth quartic curve $C \subset \PP^2_{F}$ over $F$ obtained in the same way as in Example \ref{Example: cubic surface and the description of gamma}, we have
    \[
    \gamma(C)=\alpha_2(E_X)= (-1)^{(q-1)/2+1}.
    \]
\end{examp}

\begin{examp}\label{Example: nontrivial gamma, q-1/2 is odd}
    We set
    \begin{align*}
        &u_1:= 1 + \sqrt{t}, &u_2 &:= 1 - \sqrt{t}, \\
        &u_3:= 1 + \sqrt{u}, &u_4 &:= 1 - \sqrt{u}, \\
        &u_5:= 1 + t + 2\sqrt{u}, &u_6 &:= 1 + t - 2\sqrt{u}.
    \end{align*}
    Since $u_1$ is conjugate to $u_2$, they come from a closed point
    $U_{12} \colon \Spec (E_1) \hookrightarrow \PP^2_{F}$.
    Similarly, $u_3$ and $u_4$ come from a closed point
    $U_{34} \colon \Spec (E_2) \hookrightarrow \PP^2_{F}$, and
    $u_5$ and $u_6$ come from a closed point
    $U_{56} \colon \Spec (E_2) \hookrightarrow \PP^2_{F}$.
    The $6$ points $U_i:=[u^{-2}_i, u^{-3}_i, 1] \in \PP^2_{K}$ are in a general position.
    It thus follows that the blow-up $X$ of $\PP^2_{F}$ along the closed points $U_{12}$, $U_{34}$, $U_{56}$ is naturally a smooth cubic surface $X \hookrightarrow \PP^3_{F}$ over $F$.
    By the same argument as in \cite[Section 3.4]{Lorenzo-Pirisi}, we can check that
    \[
    E_X = E^2_1 \times E^6_2 \times K^2 \times F^3.
    \]
    We note that $\alpha_{\mathrm{tot}}(F)= 1$.
    We have $\alpha_{\mathrm{tot}}(E_1)= 1 + \{ t  \}$ and $\alpha_{\mathrm{tot}}(E_2)= 1 + \{ u  \}$.
    It follows that
    $
    \alpha_{\mathrm{tot}}(E^2_1) = 1 + \{ t, t \}
    $
    and
    \[
    \alpha_{\mathrm{tot}}(E^6_2) = 1 + \{ u, u \} + \beta
    \]
    where $\beta$ is some element of degree $\geq 3$.
    We have $\alpha_1(K) = 0$ and $\alpha_2(K) = \{ t, u \} + \{ -1, tu \}$ by the same computation as in \cite[Example 3.11]{Lorenzo-Pirisi}.
    In particular, we obtain
    \[
    \alpha_{\mathrm{tot}}(K^2) = 1 + \beta'
    \]
    for some element $\beta'$ of degree $\geq 3$.
    Using the multiplicativity of the total Galois-Stiefel-Whitney classes, we obtain
    \[
    \alpha_2(E_X) = \{ t, t \} + \{ u, u \}.
    \]
    As in the previous example,
    by \cite[Chapter XIV, Section 3, Corollary]{SerreLocalfields}, we have $\{ t, t \} = (-1)^{(q-1)/2}$ and $\{ u, u \} = 1$.
    It follows that
    \[
    \alpha_2(E_X)= (-1)^{(q-1)/2}.
    \]
    Therefore, for a smooth quartic curve $C \subset \PP^2_{F}$ over $F$ obtained in the same way as in Example \ref{Example: cubic surface and the description of gamma}, we have
    \[
    \gamma(C)=\alpha_2(E_X)= (-1)^{(q-1)/2}.
    \]
\end{examp}

Our next task is to show that
smooth quartic curves $C$ over $F$ from the previous examples can be taken to have stable reduction over $\mathcal{O}_F=\Fq[[t]]$ (without enlarging $F$).
We let $u_1, \dotsc, u_6$ be as in Example \ref{Example: nontrivial gamma, q-1/2 is even} when $(q-1)/2$ is even, and let $u_1, \dotsc, u_6$ be as in Example \ref{Example: nontrivial gamma, q-1/2 is odd} when $(q-1)/2$ is odd.
We keep the notation of the previous examples.
By the result of Shioda \cite[Theorem 14]{Shioda}, we know the defining equation of $X \hookrightarrow \PP^3_{F}$.
We shall recall his result.

\begin{prop}[{\cite[Theorem 14]{Shioda}}]\label{Proposition:Shioda defining equation}
    We set
\[
c_1 := -(u_1+u_2+u_3+u_4+u_5+u_6)=-6-2t.
\]
Let $\varepsilon_m$ be the $m$-th elementary symmetric function of the $27$ elements in $K$
\[
a_i:=c_1/3 - u_i, \quad a'_i:=-2c_1/3 - u_i, \quad a''_{ij}:=c_1/3+u_i+u_j \, \, (i < j).
\]
We set
\begin{align*}
        p_2&:= \varepsilon_2/12, \\
        p_1&:= \varepsilon_5/48, \\
        q_2&:= (\varepsilon_6-168p^3_2)/96, \\
        p_0&:= (\varepsilon_8-294 p^4_2-528 p_2 q_2)/480, \\
        q_1&:= (\varepsilon_9 - 1008 p_1 p^2_2)/1344, \\
        q_0&:= (\varepsilon_{12}-608 p^2_1 p_2-4768 p_0 p^2_2-252 p^6_2-1200 p^3_2 q_2+1248 q^2_2)/17280,
\end{align*}
which make sense since $p \geq 11$.
(These elements belong to $F$, and in fact belong to $\mathcal{O}_F$.)
Then, the smooth cubic surface $X \hookrightarrow \PP^3_{F}= \Proj F[X, Y, Z, W]$
is defined by the equation
\[
f_0(X, Y, W)Z^2 + f_1(X, Y, W)Z + f_2(X, Y, W)=0
\]
where
\begin{align*}
        f_0(X, Y, W)&:= p_2 X - 2Y + q_2 W, \\
        f_1(X, Y, W)&:= p_1 XW + q_1 W^2, \\
        f_2(X, Y, W)&:= X^3 +p_0 XW^2 + q_0 W^3 - Y^2W.
\end{align*}
\end{prop}

\begin{proof}
See \cite[Theorem 14]{Shioda}.
    Although the result is stated in characteristic $0$ in \cite{Shioda}, the same result holds in characteristic $p \geq 11$, as claimed before \cite[Lemma 13]{Shioda}.
\end{proof}

One can check that none of the $27$ elements $a_i, a'_i, a''_{ij}$ are zero.
It follows from \cite[Theorem 15]{Shioda} that the point
$P:=[0, 0, 1, 0] \in X$ is not contained in any line in $X_{\overline{F}}$.
Then, as explained in Example \ref{Example: cubic surface and the description of gamma},
the morphism
\[
\mathrm{Bl}_P (X) \to \PP^2_{F}
\]
induced by the projection onto $\PP^2_{F} = (Z=0) \subset \PP^3_{F}$
expresses $\mathrm{Bl}_P (X)$ as a double covering of $\PP^2_{F}$ whose ramification locus $C \subset \PP^2_{F}$ is a smooth quartic curve.

\begin{lem}\label{Lemma:defining equation of ramified locus}
The curve
$C$ is defined by the equation
\[
g(X, Y, W):=f_1(X, Y, W)^2-4f_0(X, Y, W)f_2(X, Y, W)=0.
\]
\end{lem}

\begin{proof}
This is a well-known fact, but let us sketch the proof for the convenience of the reader.
We first note that the morphism
$\mathrm{Bl}_P (X) \to \PP^2_{F}$ maps the exceptional divisor isomorphically onto the line in $\PP^2_{F}$ defined by $f_0=0$.
The locus in $\PP^2_{F}$ where the geometric fibers of the projection $X - \{ P \} \to \PP^2_{F}$ are empty is the locus where $f_0=f_1=0$.
(The locus where $f_0=f_1=f_2=0$ is empty by the condition that $P \in X$ is not contained in any line in $X_{\overline{F}}$.)
The locus in $\PP^2_{F}$ where the geometric fibers of $X - \{ P \} \to \PP^2_{F}$ have exactly one point is the union of 
\[
\{ \, [x, y, z] \in \PP^2_{F} \, \vert \, f_0(x, y, z)=0, \, f_1(x, y, z) \neq 0 \, \}
\]
and
\[
\{ \, [x, y, z] \in \PP^2_{F} \, \vert \, f_0(x, y, z) \neq 0, \, g(x, y, z) = 0 \, \}.
\]
It follows that the ramification locus $C \subset \PP^2_{F}$ (or equivalently, the locus where the geometric fibers of
$\mathrm{Bl}_P (X) \to \PP^2_{F}$
have exactly one point) is defined by $g=0$.
\end{proof}

The polynomial $g(X, Y, W)$ lives in $\mathcal{O}_F[X, Y, W]$.
We shall prove the following claim.

\begin{lem}\label{lem: ramified locus has stable reduction}
    The closed subscheme
$\mathscr{C} \subset \PP^2_{\mathcal{O}_F}$
defined by $g(X, Y, W)=0$ is a stable curve over $\mathcal{O}_F$.
\end{lem}

\begin{proof}
We first note that $g$ is an irreducible element of $\mathcal{O}_F[X, Y, W]$ since it is irreducible in $F[X, Y, W]$ and is not divisible by the uniformizer $t$.
It follows that $\mathscr{C}$ is irreducible, which in turn implies that it is flat over $\mathcal{O}_F$.

Let $\overline{a} \in \Fq$ denote the image of an element $a \in \mathcal{O}_F$.
(We also use this notation for polynomials.)
One can check that the polynomial
$\sum^{27}_{m=0} \overline{\varepsilon}_{27-m} S^{m} \in \Fq[S]$
is equal to
\begin{equation}\label{equation:even case}
    \begin{aligned}
    S^3(S^2-1)^2&(S^2-4)(S^2-9)^2(S^2-16)(S^2-u)^2 \times \\ &(S^4-(2+2u)S^2+(1-u)^2)(S^4-(18+2u)S^2+(9-u)^2)
    \end{aligned}
\end{equation}
if $(q-1)/2$ is even, and is equal to
\begin{equation}\label{equation:odd case}
    \begin{aligned}
    S^3(S^2-9)^2&(S^2-u)^3(S^2-4u)^2(S^2-9u) \times \\ &(S^4-(18+2u)S^2+(9-u)^2)(S^4-(18+8u)S^2+(9-4u)^2)
    \end{aligned}
\end{equation}
if $(q-1)/2$ is odd.
It follows that $\overline{p}_1=\overline{q}_1=0$,
and the special fiber of $\mathscr{C}$ is defined by
\[
\overline{f_0}\overline{f_2}=(\overline{p}_2 X - 2Y + \overline{q}_2 W)(X^3 +\overline{p}_0 XW^2 + \overline{q}_0 W^3 - Y^2W)=0.
\]
The special fiber has two irreducible components $\overline{\mathscr{C}}_1$ and $\overline{\mathscr{C}}_2$: the first one $\overline{\mathscr{C}}_1$ is the line defined by
$\overline{f_0}=0$
and the second one $\overline{\mathscr{C}}_2$ is a cubic curve of Weierstrass form, defined by $\overline{f_2}=0$.
We can check that
the discriminant of $\overline{\mathscr{C}}_2$
is zero and that
the $c_4$ invariant of $\overline{\mathscr{C}}_2$
is
\begin{equation}\label{equation:c4 invariant}
    -48\overline{p}_0=
\left\{
  \begin{aligned}
  & (u-4)^2(u-16)^2/16 &  (& \text{if $(q-1)/2$ is even})\\
  & 81(u-1)^2(u-9)^2/16 &  (& \text{if $(q-1)/2$ is odd}),
  \end{aligned}
\right.
\end{equation}
which is nonzero since $\sqrt{u}$ is not contained in $\Fq$.
Therefore $\overline{\mathscr{C}}_2$ is a nodal curve; see \cite[Chapter III, Proposition 1.4]{Silverman}.
(The $c_4$ invariant of a cubic curve of Weierstrass form is defined at the beginning of \cite[Section III.1]{Silverman}.)
In order to show that $\mathscr{C}$ is a stable curve, it suffices to show that
$(\overline{\mathscr{C}}_1)_{\Fbar}$
intersects
with
$(\overline{\mathscr{C}}_2)_{\Fbar}$
in exactly three points.
Since they do not intersect each other in $W=0$, it is enough to show that
the discriminant of
the polynomial
\begin{equation}\label{equation:intersection}
x^3 + \overline{p}_0 x +  \overline{q}_0 -(\frac{\overline{p}_2}{2} x + \frac{\overline{q}_2}{2})^2=x^3 -\frac{\overline{p}^2_2}{4} x^2 + (\overline{p}_0 - \frac{\overline{p}_2\overline{q}_2}{2})x +  \overline{q}_0 -\frac{\overline{q}^2_2}{4}
\end{equation}
is nonzero.
The discriminant is equal to
\begin{equation}\label{equation:discriminant}
    \left\{
  \begin{aligned}
  & 81u^2(u-1)^2(u-9)^2/64 &  (& \text{if $(q-1)/2$ is even})\\
  & 729u^6(u-9)^2(u-9/4)^2/16 &  (& \text{if $(q-1)/2$ is odd}),
  \end{aligned}
\right.
\end{equation}
which is nonzero since $\sqrt{u}$ is not contained in $\Fq$.
\end{proof}

\begin{remark}
\begin{enumerate}
    \item One can compute the discriminant of $\overline{\mathscr{C}}_2$, the $c_4$ invariant of $\overline{\mathscr{C}}_2$, and the discriminant
    of (\ref{equation:intersection}) by using a computer algebra software.
    More precisely, we consider the polynomial
    $Q(u, S)$ with coefficients in $R:=\ZZ[1/(2 \cdot 3 \cdot 5 \cdot 7)]$ defined as in (\ref{equation:even case}) or (\ref{equation:odd case}) (depending on whether $(q-1)/2$ is even or odd).
    Here we regard $u$ as a variable.
    Let $\varepsilon_m \in R[u]$ be the coefficient of $S^{27-m}$ in $Q(u, S)$, and we define $p_2, p_1, q_2, p_0, q_1, q_0 \in R[u]$ in the same way as above.
    Then one can check that in the ring $R[u]$, the discriminant of the elliptic curve over $R[u]$ defined by
    $X^3 +p_0 XW^2 + q_0 W^3 - Y^2W$ is zero, and its $c_4$ invariant has the form of (\ref{equation:c4 invariant}).
    Moreover one can show that
    the discriminant of the cubic polynomial with coefficients in $R[u]$ defined in the same way as (\ref{equation:intersection}) has the form of (\ref{equation:discriminant}).
    
    \item We can more easily prove the following weaker statement, which is in fact enough for our purpose (i.e.\ the proof of Theorem \ref{Theorem: g=3 in characteristic 0}): For a large enough $q$, there exists an element $u \in \Fq^\times$ such that $\sqrt{u}$ is not contained in $\Fq$
    and the closed subscheme
$\mathscr{C} \subset \PP^2_{\mathcal{O}_F}$
as in Lemma \ref{lem: ramified locus has stable reduction} is a stable curve over $\mathcal{O}_F$.
To see this, we remark that in fact the vanishing of the discriminant of $\overline{\mathscr{C}}_2$
is not necessary
once we know that the $c_4$ invariant is nonzero
(since this implies that $\overline{\mathscr{C}}_2$ is a smooth curve or a nodal curve).
One can check that
the $c_4$ invariant and the discriminant of (\ref{equation:intersection})
have nonzero polynomial expressions in $u$ with coefficients in $\Fq$.
Since the degrees of these polynomials are bounded uniformly for $q$,
it follows that for a large enough $q$,
there exists an element $u \in \Fq^\times$ such that $\sqrt{u}$ is not contained in $\Fq$ and $u$ is not a root of any of these polynomials.
This proves our claim.
\end{enumerate}
\end{remark}

In total, we have shown the following statement:

\begin{lem}\label{Lemma: stable curve with nontrivial Brauer class}
There exists a smooth quartic curve $C \subset \PP^2_{F}$ over $F=\Fq((t))$ with the following properties:
    \begin{itemize}
        \item $\gamma(C) \in \Br(F)[2]$ is nonzero.
        \item $C$ has stable reduction over $\mathcal{O}_F=\Fq[[t]]$.
    \end{itemize}
\end{lem}

\begin{proof}
    This follows from Example \ref{Example: nontrivial gamma, q-1/2 is even}, Example \ref{Example: nontrivial gamma, q-1/2 is odd}, and Lemma \ref{lem: ramified locus has stable reduction}.
\end{proof}

\begin{remark}
    We found smooth quartic curves as in Lemma \ref{Lemma: stable curve with nontrivial Brauer class} through trial and error.
    It would be interesting to give a more conceptual way to construct such smooth quartic curves.
\end{remark}

\begin{cor}\label{Corollary: g=3 in positive characteristic}
    Let $p$ be a prime such that $p \geq 11$.
    \begin{enumerate}
        \item $\Br(\overline{\mathcal{M}}_{3, \Fq})[2]=0$ for any finite field $\Fq$ of characteristic $p$.
        \item $\Br(\overline{\mathcal{M}}_{3, \Fbar})[2]=0$.
    \end{enumerate}
\end{cor}

\begin{proof}
    (1) By \cite{Lorenzo-Pirisi}, we have $\Br(\mathcal{M}_{3, \Fq})[2] \simeq \ZZ/2\ZZ$.
    Let $\gamma \in \Br(\mathcal{M}_{3, \Fq})[2]$ be the nonzero element.
    It suffices to show that $\gamma$ does not lie in the image of $\Br(\overline{\mathcal{M}}_{3, \Fq}) \hookrightarrow \Br(\mathcal{M}_{3, \Fq})$.
    This follows from Lemma \ref{Lemma: stable curve with nontrivial Brauer class} since $\Br(\mathbb{F}_{q}[[t]])=0$.

    (2) This follows from (1).
\end{proof}

We can now prove the desired statement:

\begin{thm}\label{Theorem: g=3 in characteristic 0}
    We have $\Br(\overline{\mathcal{M}}_{3, \overline{\QQ}})=0$.
\end{thm}

\begin{proof}
    By \cite{Lorenzo-Pirisi},
    we have $\Br(\mathcal{M}_{3, \overline{\QQ}}) \simeq \ZZ/2\ZZ$.
    Since $\Br(\overline{\mathcal{M}}_{3, \overline{\QQ}})$ is embedded into $\Br(\mathcal{M}_{3, \overline{\QQ}})$,
    by virtue of Remark \ref{Remark:Constancy of Picard scheme of moduli space of stable curves} and Lemma \ref{Lemma: vanishing of Brauer class and reduction}, it suffices to show that $\Br(\overline{\mathcal{M}}_{3, \Fbar})[2]=0$
    for some odd prime $p$.
    This follows from Corollary \ref{Corollary: g=3 in positive characteristic}.
\end{proof}
\subsection{Genus greater than or equal to four}

We conclude this paper by recalling the following well-known vanishing for the case $g \geq 4$, together with a rough explanation of where the condition that $g \geq 4$ is used.

\begin{thm}[Korkmaz-Stipsicz]\label{thm vanishing over C}
Assume that $g\geq 4$ and $n$ arbitrary.
Then
$
\Br(\overline{\mathcal{M}}_{g,n,\Qbar})=0.
$
\end{thm}
\begin{proof}
As in the proof of Proposition \ref{Prop: Vanishing (1,1) bar Q}, it is enough to show that $\Br(\mathcal{M}_{g,n,\mathbb{C}})=0.$
This vanishing is explained in \cite[Theorem 4.1]{Fringuelli-Pirisi} and is due to Korkmaz-Stipsicz and Harer:
The key point is that 
$$
H^{3}(\mathcal{M}_{g,n,\mathbb{C}}^{\text{an}},\ZZ)_{\tors}=0,
$$
which reduces to a computation in group homology and this is \cite[Theorem 1.1]{SecondHomologyGenusGEQ4} (where they rely crucially on a stablization result due to Harer \cite{HarerStability}, where the restriction $g\geq 4$ enters).
\end{proof}

\subsection*{Acknowledgments}
The authors would like to thank K\k{e}stutis \v{C}esnavi\v{c}ius and Jochen Heinloth for very helpful discussions.
The first author (S.B.) wants to thank Stefan Schröer for giving an interesting talk on his work \cite{SchroerEnriquesZ} in Essen, which inspired some arguments in Section 3.
The authors both acknowledge financial supports by DFG RTG 2553 and JSPS KAKENHI Grant Number 24H00015.
The work of the second author (K.I.) was also supported by JSPS KAKENHI Grant Number 24K16887.

\bibliographystyle{abbrvsort}
\bibliography{bibliography.bib}
\end{document}